\documentclass[12pt]{article}
\usepackage[utf8]{inputenc}
\usepackage[T1]{fontenc}
\usepackage{lmodern}
\usepackage{amsmath,amssymb,amsthm}
\usepackage[margin=2.5cm]{geometry}
\usepackage[dvipsnames]{xcolor}
\usepackage{hyperref}
\hypersetup{
  colorlinks=true,
  linkcolor=MidnightBlue,
  citecolor=MidnightBlue,
  urlcolor=MidnightBlue,
  pdftitle={Almost Golomb Sequences},
  pdfauthor={Beno\^it Cloitre and Gandhar Joshi}
}
\usepackage{enumitem}
\usepackage{graphicx}
\newtheorem{theorem}{Theorem}[section]
\newtheorem{lemma}[theorem]{Lemma}
\newtheorem{proposition}[theorem]{Proposition}
\newtheorem{corollary}[theorem]{Corollary}
\theoremstyle{definition}
\newtheorem{definition}[theorem]{Definition}

\theoremstyle{remark}

\newtheorem{remark}[theorem]{Remark}

\newcommand{\Z}{\mathbb{Z}}
\newcommand{\Ical}{\mathcal{I}}
\newcommand{\one}{\mathbf{1}}
\newcommand{\oeis}[1]{\href{https://oeis.org/#1}{\textbf{#1}}}

\title{Almost Golomb Sequences}
\author{Beno\^it Cloitre and Gandhar Joshi}
\date{}

\begin{document}
\maketitle

\begin{abstract}
Golomb's sequence $(G(n))$ is the unique nondecreasing sequence of positive integers
in which each $n$ appears exactly $G(n)$ times. It satisfies the global
self-referential rule
\[
  G\bigl(G(n)+G(n-1)+\cdots+G(1)\bigr)=n,
\]
grows smoothly like a power of $n$ governed by the golden ratio,
and is not $k$-regular for any $k\ge 2$.

We introduce almost Golomb sequences, obtained by truncating the
cumulative sum to a fixed size sliding window $r$,
\[
  a\bigl(a(n)+a(n-1)+\cdots+a(n-r+1)\bigr)=n.
\]
This finite-memory truncation completely changes the nature of the sequence.
The smooth power law gives way to oscillatory linear growth, and the
sequence becomes $r$-regular for every $r\ge 2$. For small values of $r$
we establish explicit denesting formulas, prove that $a(n)/n$ does not
converge, and reveal the combinatorial structure including a cellular
automaton and a palindromic substitution.

When one varies $r$, the Golomb sequence itself reappears. We prove that
the first $r$ terms of the order-$r$ sequence form a shifted copy of
Golomb's sequence, and that this initial segment controls the maximum
multiplicity across the whole family. For $r\ge 5$ in particular, the maximal
multiplicity is exactly $G(r-1)$. The sequence that was truncated returns
as the law governing the family it generated.
\end{abstract}



\noindent(Concerned with sequences \oeis{A001462}, \oeis{A001463},
\oeis{A002487}, \oeis{A002516}, \oeis{A002517}, \oeis{A003605},
\oeis{A079000}, \oeis{A080637}, \oeis{A394217}, \oeis{A394218},
and \oeis{A394219}.)

\section{Introduction}

\subsection{Golomb's sequence}\label{sec:golomb}


Golomb's sequence $(a(n))_{n\ge 1}$ is the unique nondecreasing
sequence of positive integers\footnote{with clear context of its index range, we also use $(a(n))$ to denote the sequence in question.} in which the integer $n$ appears exactly $a(n)$ times.
It begins
\[
  1, 2, 2, 3, 3, 4, 4, 4, 5, 5, 5, 6, 6, 6, 6, 7, 7, 7, 7, \ldots
\]
In particular, $a$ satisfies
\begin{equation}\label{eq:golomb}
  a\!\left(\sum_{k=1}^{n} a(k)\right) = n, \qquad n\ge 1,
\end{equation}
a formulation recorded in the OEIS entry \oeis{A001462} \cite{OEIS}.
Golomb's sequence is one of the most studied self-referential sequences in the
world of integer sequences, introduced by Golomb~\cite{Golomb66} and prominently
featured by Sloane among his favorite sequences.


Its asymptotic behavior is smooth:
\begin{equation}\label{eq:golomb-asymp}
  a(n) \sim \varphi^{2-\varphi}\, n^{\varphi-1}, \qquad n\to\infty,
\end{equation}
where $\varphi=(1+\sqrt{5})/2$ is the golden ratio~\cite{petermann95},
and $f\sim g$ means $f(n)/g(n)\to 1$ as $n\to\infty$.
Since $a(n)$ grows like $n^{\varphi-1}$ with irrational algebraic exponent
$\varphi-1\approx 0.618$, we immediately conclude the following:
\begin{proposition}
    Golomb's sequence is not $k$-regular for any integer $k\ge 2$.
\end{proposition}
\begin{proof}
   If Golomb's sequence were $k$-regular, its subsequences along geometric progressions $(a(k^m))_{m\ge 0}$ would satisfy a linear recurrence with integer coefficients~\cite[Theorem~2.2]{AS92}, which would require the asymptotic growth factor $k^{\varphi-1}$ to be an algebraic integer. But by the Gelfond--Schneider theorem \cite[Theorem~2.4]{Baker1975}, $k^{\varphi-1}$ is transcendental whenever $k$ is algebraic, $k\notin\{0,1\}$, and $\varphi-1$ is an irrational algebraic number, leading to a contradiction.
\end{proof}

The combinatorial depth of Golomb's sequence has recently been made explicit by
Claes and Miyamoto~\cite{CM26}, who define the \emph{golombic operator}
$\Gamma$: given a sequence $a$, $\Gamma(a)$ is the sequence whose $n$-th
run has length $a(n)$ (that is, value $n$ is repeated $a(n)$ times).
They show that Golomb's sequence is the unique fixed point of $\Gamma$
in the class of nondecreasing sequences, establishing a rich algebraic
structure. Our almost Golomb sequences occupy a different regime (Definition~\ref{def:almost-golomb}). They are not fixed
points of $\Gamma$, but their run lengths are bounded (depending on $r$) and the
corresponding run-length sequence is {\em $r$-automatic}. For $r=2$ in particular the
run lengths lie in $\{1,2\}$ (Proposition~\ref{prop:N2_multiplicity}).
Numerically, the iterates $\Gamma^k(a_r)$ appear to converge to $G$ in the
prefix topology, suggesting a deeper dynamical link that we leave for future work. For a primer on automatic sequences, see Allouche and Shallit \cite{AS03}.

\subsection{Almost Golomb sequences}

What happens when one replaces the full cumulative sum in
Equation~\ref{eq:golomb} by a sliding window of the $r$ most recent terms? This leads to
sequences that we call \emph{almost Golomb sequences}: they share the meta-recursive
spirit of Golomb's sequence~\cite[p.~138]{Hofstadter}, but are built from a finite-memory rule.

\begin{definition}\label{def:almost-golomb}
  Let $r\ge 2$. The \emph{almost Golomb sequence of order $r$} is the
  nondecreasing sequence of positive integers $(a(n))_{n\ge 1}$,
  extended by $a(k)=0$ for $k\le 0$, satisfying
  \begin{equation}\label{eq:def}
    a\bigl(a(n)+a(n-1)+\cdots+a(n-r+1)\bigr) = n, \qquad n\ge 1,
  \end{equation}
  with $a(n)$ minimal: for every $n\ge 2$, no integer $m$ with
  $a(n-1)\le m<a(n)$ satisfies
  $a(m + a(n-1)+\cdots+a(n-r+1)) = n$.
\end{definition}
The position $a(n)+\cdots+a(n-r+1)$ at which the constraint is imposed generally lies beyond $n$, so the definition is implicit.
The existence and uniqueness of the sequence, as well as the
consistency of the greedy rule, are established in Theorem~\ref{thm:global_run}. 
\begin{remark}
We require minimality in Definition~\ref{def:almost-golomb}, as without it, the self-referential rule (Equation~\ref{eq:def}) yields more than one satisfying sequence. One such sequence is the {\em inhomogeneous Beatty} sequence given by\[b(n):=\left\lfloor\frac{n+1}{\sqrt{2}}\right\rfloor,\qquad n\geq1,\] which is also an `almost Golomb sequence of order $2$' but without the minimality condition; see \cite[Theorem 1]{Cloitre26Beatty}. In particular, for $r=2$, the greedy $a(n)$ (from Section~\ref{sec:r2}) and the non-greedy Beatty sequence $b(n)$ differ at least at $n=12$.
\end{remark}

The almost Golomb sequences belong to a broader family of monotone self-referential
sequences admitting an automatic structure via denesting. A comparative table
of related sequences is collected in Appendix~\ref{app:nested}.

This truncation from an infinite to a finite memory window changes the nature of the sequence.
Golomb's global rule forces a smooth, sublinear growth governed by an irrational
exponent. The finite-window rule instead imposes linear growth accompanied by
discrete, $r$-automatic oscillations.
The sequence transitions from existing in the realm of analytic power laws to that of
automatic sequences.

A sequence $(f(n))_{n\ge 0}$ is \emph{$r$-regular} (in the sense of Allouche and Shallit~\cite{AS92})
if all its subsequences of the form $(f(r^i n+j))_{n\ge 0}$ lie in a finitely
generated $\mathbb{Z}$-module. Informally, its values along arithmetic progressions
in base $r$ satisfy linear recurrences.
We refer to these as \emph{divide-and-conquer} rules: they split the
sequence along arithmetic progressions in base $r$, and each piece is
expressed through simpler pieces. When the $r$-kernel is finite the rules
can be listed explicitly. The Thue--Morse sequence, for instance, is
determined by $t(2n)=t(n)$ and $t(2n+1)=1-t(n)$, and any further rule it
satisfies, such as $t(4n)=t(n)$, is a consequence of these two. The
framework of $r$-regular sequences is the natural setting for such rules;
see \cite[Chapter 16]{AS03}.

The paper is organised as follows.
Section~\ref{sec:general} establishes $r$-regularity universally.
Sections~\ref{sec:r2}--\ref{sec:r3} give complete constructive proofs
for $r\in\{2,3\}$.
Sections~\ref{sec:r4}--\ref{sec:r5} treat $r\in\{4,5\}$ via explicit
recurrences and \texttt{Walnut} certification.
Section~\ref{sec:nonconv} proves non-convergence of $a(n)/n$.
Section~\ref{sec:combinatorial} presents combinatorial interpretations
for $r\in\{2,3\}$.
Sections~\ref{sec:golomb_return}--\ref{sec:open} present the Golomb
meta-structure and open questions.

The main contributions of this paper are:
\begin{enumerate}
  \item A proof (Theorem~\ref{thm:structural_regularity}) that $(a(n))$ is $r$-regular
        for every $r\ge 2$, by reducing the automaticity of the first-difference
        sequence given by $d(n)=a(n+1)-a(n)$ to a finite window of local data.
  \item Explicit denesting of the almost Golomb sequence for each
        $r\in\{2,3,4,5\}$ into $r$-adic divide-and-conquer formulas
(Theorems~\ref{thm:r2},~\ref{thm:r3},~\ref{thm:r5} and Corollary~\ref{cor:r4}).
        The cases $r\in\{2,3\}$ are proved directly, and the cases $r\in\{4,5\}$
        use a hybrid method combining explicit recurrence identification
        with automatic certification by \texttt{Walnut}~\cite{Shallit2022}.
        The non-generic terms are controlled by explicit $r$-automatic
        binary correction sequences: one for $r=3$, four
        for $r=4$, and a disjoint pair $(\varepsilon,\eta)$ for $r=5$.
  \item For $r=2$, a Mallows-type nested recurrence
        (Proposition~\ref{prop:mallows-r2}):
        \[
          a(n+1)=1+a\bigl(n+1-a(a(n)+1)+a(a(n)-1)\bigr),
        \]
        analogous to the Mallows recurrence $G(n+1)=1+G(n+1-G(G(n)))$
        for Golomb's sequence $G(n)$.
  \item For each $r\in\{2,3,4,5\}$, a proof that $a(n)/n$ does not converge,
        identifying two explicit rational limit points
        (Theorems~\ref{thm:ratio2},~\ref{thm:ratio3},
       ~\ref{thm:ratio4},~\ref{thm:ratio5}).
        For $r\in\{4,5\}$ the proofs use
        Lemma~\ref{lem:auto-geometric} to extract the periodic
        corrector pattern along geometric progressions from the DFAO.
  \item Combinatorial interpretations for the cases $r=2$ and $r=3$
        (Section~\ref{sec:combinatorial}): a second-bit characterization of
        multiplicities and a Boolean cellular automaton for $r=2$,
        and a palindromic substitution rule for $r=3$.
  \item The Golomb meta-structure (Section~\ref{sec:golomb_return}).
      We prove that the first $r$ terms of the order-$r$ sequence are a
      shifted copy of Golomb's sequence, $a_r(n)=G(n-1)$ for $3\le n\le r$
      (Theorem~\ref{thm:prefix}), and that for $r\ge 5$ the maximal
      multiplicity $M(r)=\sup_n N_r(n)$ equals $G(r-1)$
      (Theorem~\ref{thm:domination}). Consequently, if $j_k$ denotes the
      smallest order $r$ with $M(r)\ge k$, then
      \[
        j_{k+1}-j_k = G(k), \qquad j_k = S(k-1)+2,
      \]
      where $S(k)=\sum_{n\le k}G(n)$. As a by-product,
      $L_{\max}(r)=\limsup a_r(n)/n$ tends to $0$ at rate $O(r^{\varphi-2})$.
      These two theorems were stated as conjectures in the first version of
      this paper and were first proved by Zhou~\cite{Zhou26}.
\end{enumerate}

\section{Universal \texorpdfstring{$r$}{r}-regularity}\label{sec:general}

We now prove that the almost Golomb sequence of order $r$ is
$r$-regular for every $r\ge 2$. The proof has two parts.

First, we establish the {\em global run structure} of the sequence, i.e., we conclude that every natural number appears once, the increments are always $0$ or $1$, and for every $m\ge 3$
the first occurrence of $m$ is at position, given as a {\em window sum}
\begin{equation}\label{EQ:windowSum}
    S_m=S_m^{(r)}:=a(m)+a(m-1)+\cdots+a(m-r+1).
\end{equation}
Second, we use this run structure to express the
difference sequence at scale $rn+i$ in terms of a finite local window,
and then apply the Allouche--Shallit criterion (Theorem~\ref{thm:AS}) to prove the automaticity of $(d(n))$, which consequently proves the regularity of $(a(n))$  (Theorem~\ref{thm:structural_regularity}).

\subsection{Global run structure}

We extend the sequence by setting $a(k)=0$ for $k\le 0$.
For $n\ge 1$, and $S_n$ as given in Equation~\ref{EQ:windowSum},
the defining constraint (Equation~\ref{eq:def}) now reads $a(S_n)=n$. We first explain the inevitable initial values.
\begin{lemma}\label{LEM:InitValues}
    The initial block of an almost Golomb sequence $(a(n))$ of any order $r\geq2$ is inevitably given by $a(1)=1$ and $a(q)=2$ for $2\le q<S_3$.
\end{lemma}
\begin{proof}
First, if $a(1)=0$, then $S_1=a(1)+a(0)+\cdots+a(2-r)=0$, and the defining constraint $a(S_1)=1$ would read $a(0)=1$, contradicting the convention $a(k)=0$ for $k\le 0$. Second, if $a(1)=c\ge 2$, then $S_1=c$ and the constraint reads $a(c)=1$. But $(a(n))$ is nondecreasing and $c\ge 1$, so $a(c)\ge a(1)=c\ge 2$, a contradiction. Hence $a(1)=1$.\\
    Next, we must have $a(2)=2$. Otherwise, if \(a(2)=1\), then
\(S_2=a(2)+a(1)=2\) and \(a(S_2)=a(2)=1\), not following Definition~\ref{def:almost-golomb}. Similarly, if \(a(2)\ge 3\),
then \(a(S_2)\ge a(2)\ge 3\), which again is impossible.\\
Since \(S_2=a(2)+a(1)=3\), the equation \(a(S_2)=2\) forces
\[
a(3)=2.
\]
Next,
\[
S_3=a(3)+a(2)+a(1)+a(0)+\cdots+a(4-r).
\]
Thus
\[
S_3=
\begin{cases}
4, & r=2,\\
5, & r\ge 3.
\end{cases}
\]
So the initial block is
\[
a(1)=1,\qquad a(q)=2\quad(2\le q<S_3).
\]
\end{proof}
\begin{lemma}\label{lem:anchor-growth}
For every $r\ge 2$, the sequence of starting positions $(S_n)_{n\ge 1}$ is strictly increasing.
\end{lemma}

\begin{proof}
From Definition~\ref{def:almost-golomb}, $a(n)$ is nondecreasing. If $S_n\le S_{n-1}$ for some $n\ge 2$, then $n=a(S_n)\le a(S_{n-1})=n-1$,
a contradiction.
\end{proof}

By {\em run} of a value $t$, we mean the consecutive positions in a sequence at which $t$ appears. We will use the notation $[x,\; x+y]$ to show the run of a value (there shall only be one for any value in an almost Golomb sequence!), where we let $[x,\; x+y]:=x,x+1,x+2,\dots,x+y-1,x+y$ for some integers $x\le y$. The following theorem describes the run structure of $a$, concluding the first part required for the universal $r$-regularity proof.

\begin{theorem}\label{thm:global_run}
For every $r\ge 2$, there exists a unique almost Golomb sequence of order $r$.
It satisfies:
\begin{enumerate}[label=(\alph*)]
  \item For all $n\ge 1$, the first difference $d(n)=a(n+1)-a(n)$ is always either $0$ or $1$, and hence, every positive integer occurs at least once;
  \item for every $m\ge 3$, the first occurrence of $m$ is exactly at $S_m$, and the run of value $m$ is exactly $[S_m,\;S_{m+1}-1]$.
\end{enumerate}
\end{theorem}

\begin{proof}
(a) We construct the sequence by runs, show that the listed properties hold, and finally show its uniqueness. First, we invariably take from
Lemma~\ref{LEM:InitValues},
the initial block
\[
a(1)=1,\qquad a(q)=2\quad(2\le q<S_3).
\]

This is followed by an inductive argument. For some \(m\ge 3\), suppose that the nondecreasing sequence $(a(n))$ has been constructed up to position \(S_m-1\), the first differences $d(n)\in\{0,1\}$, the runs of any value have length at most \(r\), and $S_m\ge m+1$.

We now start the run of value \(m\) by setting
\[
a(S_m)=m.
\]
The length of the run of $m$ is forced. Indeed,

\begin{equation}\label{EQ:runlength}
    S_{m+1}-S_m=\sum_{j=0}^{r-1} a(m+1-j) - \sum_{j=0}^{r-1} a(m-j) = a(m+1)-a(m+1-r).
\end{equation}
Let $L_m$ denote this difference of consecutive window sums. We claim that \(L_m\ge 1\).

If \(L_m=0\), then monotonicity forces \((a(n))\) to be constant on the
whole interval $[m+1-r,m+1]$.
Let this common value be \(v\). 

For values \(v<m\), whose runs are already present in the hypothesised prefix of $(a(n))$, the run length at least \(r+1\), contradicting our assumption. Now, for the run currently under construction, i.e., if \(v=m\), then \(m\) already occurs at \(m+1-r\). Since the first occurrence of \(m\) is being placed at
\(S_m\), we would have
$m+1-r\ge S_m$.
But that contradicts our hypothesis, $S_m\ge m+1$, since \(r\ge 2\).  Therefore \(L_m\ge 1\). 

This also proves that $S_m\ge m+1$ for all $m\ge 3$, since
\begin{equation*}\label{EQ:S_mGEm+1}
    S_{m+1}=S_m+L_m\ge S_m+1 \ge (m+1)+1=m+2.
\end{equation*}
This implies $a(S_m)\ge a(m+1)$; and since $(a(n))$ is nondecreasing, we get
\[a(S_m)\ge a(m+1)\ge a(m).\]
It means that if we prove $d(S_m-1)=a(S_m)-a(S_m-1)\le 1$, then $d(m)\le 1$ follows directly from the hypothesis. Our hypothesis provides $a(n)$ upto $n=S_m-1$. But we also have the next term $a(S_m)=m$. 
Now,
\[m=a(S_m)\ge a(S_m-1)\ge a(S_{m-1})=m-1.\]
So, $a(S_m-1)\ge m-1$, which implies $d(S_m-1)\le 1$, which concludes that $d(m)\le 1$.\\\\
(b) We are now ready to prove the constraints given in the theorem statement. We have,
\[
L_m=a(m+1)-a(m+1-r)=\sum_{j=0}^{r-1} a(m+1-j) - \sum_{j=0}^{r-1} a(m-j)=\sum_{j=0}^{r-1}d(m-j)\le r.
\]
Since each \(1\le L_m\le r\), the induction continues indefinitely, and we can prescribe the run of $m$ as \[
a(q)=m\qquad(S_m\le q<S_m+L_m),
\]
where for \(S_{m+1}=S_m+L_m\), the run of $m$ is exactly $[S_m,S_{m+1}-1]$.

We have yet only proved that such a run structure for all \(m\ge 3\) exists, and that no positive integer is skipped, i.e., \(d(n)\in\{0,1\}\).
We still need to show that $(a(n))$ is unique. Let \(q\ge 3\), and write
$m:=a(q)$.
For the candidate \(m\), we have
\[
m+a(q-1)+\cdots+a(q-r+1)
=
a(q)+a(q-1)+\cdots+a(q-r+1)
=
S_q,
\]
and \(a(S_q)=q\). Thus \(m\) is admissible.

If \(q\) is not the first position of its run, then $a(q-1)=m$, so no candidate \(<m\) is allowed.
If \(q\) is the first position of the run of \(m\), then $a(q-1)=m-1$.
The only possible smaller candidate is \(m-1\). For this value, the
window sum is $S_q-1$.
But \(S_q\) is the first occurrence of \(q\), so $a(S_q-1)<q$.
Thus \(m-1\) is not admissible. Hence \(m\) is the smallest admissible
value in all cases.

Finally, the construction is deterministic. At each step, once the
previous prefix is known, both \(S_m\) and
\[
L_m=a(m+1)-a(m+1-r)
\]
are forced. Therefore the greedy sequence is unique.
\end{proof}

Composing the map $n\mapsto S_n$ with itself gives the following identity.

\begin{lemma}\label{lem:nested-anchor}
For every $r\ge 2$ and $n\ge 1$, 
\[S_{S_n}=rn-R_n\text{, where }R_n=\sum_{j=1}^{r-1}(r-j)d(S_n-j).\] Moreover, $0\le R_n\le \frac{r(r-1)}2$.
\end{lemma}

\begin{proof}
Since \(a(S_n)=n\), and $a(S_n-1)=a(S_n)-d(S_n-1)$, $a(S_n-2)=a(S_n)-d(S_n-1)-d(S_n-2)$, and so forth, for \(0\le k\le r-1\),
\[
a(S_n-k)
=
n-\sum_{j=1}^{k}d(S_n-j).
\]
Summing over \(k=0,\ldots,r-1\), we get
\begin{equation}\label{EQ:SSN}
   S_{S_n}
=
\sum_{k=0}^{r-1}a(S_n-k)
=
rn-\sum_{k=0}^{r-1}\sum_{j=1}^{k}d(S_n-j)=rn-\sum_{j=1}^{r-1}(r-j)d(S_n-j), 
\end{equation}
as stated. The bound follows with \(d(S_n-j)\) being either \(0\) or \(1\) for all $j$.
\end{proof}

Thus we have the translation of the self-referential structure of \((a(n))\) into
statements about the binary difference sequence \((d(n))\). The $S_{S_n}$ formula (Equation~\ref{EQ:SSN}) shows that \(S_{S_n}\) is always within a bounded distance of \(rn\), with a correction \(R_n\) depending only on the last \(r-1\) differences before \(S_n\). 

The next lemma is the local ``decoder''. It tells us that if we know the
nearby run lengths, then we know where the jumps occur near
\(S_n\), which eventually allows us to determine $d(rn+i)$ with a bounded window of local data.

\begin{lemma}\label{lem:local_projection}
Let \(A\ge 3\).

\begin{enumerate}[label=(\alph*)]
\item For \(X\ge 0\), the value \(d(S_A+X)\) is determined by $[L_A, L_{A+X}]$. More precisely,
\[
d(S_A+X)=1\quad \text{if and only if}\quad \sum_{j=0}^{\delta-1}L_{A+j}=X+1
\]
for some \(\delta\in\{1,\ldots,X+1\}\).

\item For \(Y\ge 1\), the value \(d(S_A-1-Y)\) is determined by $[L_{A-1}, L_{A-Y}]$.
More precisely,
\[
d(S_A-1-Y)=1\quad\text{if and only if}\quad\sum_{j=1}^{\delta}L_{A-j}=Y
\]
for some \(\delta\in\{1,\ldots,Y\}\).

\item Finally,
\[
d(S_A-1)=1.
\]
\end{enumerate}
\end{lemma}

\begin{proof}
Let \(A\ge 3\). By the run structure (Theorem~\ref{thm:global_run}), for every $k=S_A-1$, $d(k)=1$.

$(a)$ For \(X\ge 0\), the condition \(d(S_A+X)=1\) is equivalent to
\[
S_A+X+1=S_{A+\delta}
\]
for some \(\delta\ge 1\). This is equivalent to
\[
X+1=S_{A+\delta}-S_A
=
\sum_{j=0}^{\delta-1}L_{A+j}.
\]
Since every \(L_u\ge 1\), necessarily \(\delta\le X+1\).

$(b)$ The proof for the left side of the local window is identical. The condition
\(d(S_A-1-Y)=1\) is equivalent to
\[
S_A-Y=S_{A-\delta}
\]
for some \(\delta\ge 1\), and hence 
\[Y=S_A-S_{A-\delta}
=
\sum_{j=1}^{\delta}L_{A-j}.
\]
Again \(\delta\le Y\).

$(c)$ Finally, \(d(S_A-1)=1\), because \(S_A\) is the
first position of the run of \(A\).
\end{proof}
\subsection{The Allouche--Shallit criterion}
A sequence $(f(n))_{n\ge 0}$ with values in a finite set is \emph{$r$-automatic}
if its $r$-kernel $\{(f(r^in+j))_{n\ge 0}:i\ge 0,\,0\le j<r^i\}$ is finite \cite[Proposition V.3.3]{Eilenberg1974}.
A sequence with integer values is \emph{$r$-regular}~\cite{AS92} if the $\mathbb{Z}$-module
spanned by its $r$-kernel is finitely generated. Every $r$-automatic sequence is $r$-regular, and the converse is not true.

The following criterion underlies all automaticity results in this paper.

\begin{theorem}[Theorem 2.2, \cite{AS12}]\label{thm:AS}
Let $(U(n))_{n\ge 0}$ take values in a finite set, and let $r\ge 2$.
Then $U$ is $r$-automatic if there exist $a,b\ge 0$ and $n_0\ge 0$ such that
for all $n\ge n_0$ and $0\le i<r$,
\[
  U(rn+i) = f_i\bigl(U(n-a),\ldots,U(n+b)\bigr)
\]
for some functions $f_i$.
\end{theorem}

\subsection{Universal $r$-regularity}

Set $d(0):=1$, so that $a(n)=\sum_{k=0}^{n-1}d(k)$ for all $n\ge 1$.

\begin{theorem}\label{thm:structural_regularity}
For every $r\ge 2$, $(d(n))$ is $r$-automatic and $(a(n))$ is $r$-regular.
\end{theorem}

\begin{proof}
Set \[C:=2r-3,\; D:=\frac{r(r+1)}{2}-1,\; n_0:=\max(r+1,2r-3),\; 0\le i<r,\;\text{ and }n\ge n_0.\]
Following the approach of Theorem~\ref{thm:AS}, we shall prove that \(d(rn+i)\) for every residue $i$ is determined by the finite binary local window
\[
d(n-C),d(n-C+1),\ldots,d(n+D)
\]
for all \(n\).

First, we replace $rn+i$ by a position near $S_{S_n}$. 
By Lemma~\ref{lem:nested-anchor}, 
\[S_{S_n}=rn-R_n\text{ with }0\le R_n\le \frac{r(r-1)}{2}.\]
Hence $rn+i=S_{S_n}+X,$ where $X=R_n+i$.
Since \(0\le i\le r-1\), we have
\[
0\le X\le \frac{r(r-1)}2+(r-1)
=
\frac{r(r+1)}2-1
=
D.
\]
Thus, $d(rn+i)=d(S_{S_n}+X)$, with $0\le X\le D$. By Lemma~\ref{lem:local_projection}, the value
\(d(S_{S_n}+X)\) is determined by 
$[L_{S_n}, L_{S_n+D}]$.

Now lets express these run lengths through differences near
\(S_n\). For \(0\le k\le D\), the run-length formula gives
\[
L_{S_n+k}
=
\sum_{j=0}^{r-1}d(S_n+k-j).
\]
Therefore we need to know $d(S_n+x)$ for all integers \(x\) in the interval $-(r-1)\le x\le D$.

For $x\ge 0$, using Lemma~\ref{lem:local_projection}$(a)$ with $A=n$, we get that
\(d(S_n+x)\) is determined by $[L_n,L_{n+x}]$.
Since \(x\le D\), these are among $[L_n, L_{n+D}]$. For $x=-1$, $d(S_n-1)=1$.

For $x=-1-Y$ with $1\le Y\le r-2$, Lemma~\ref{lem:local_projection}$(b)$ with $A=n$ shows that \(d(S_n-1-Y)\) can be determined by $[L_{n-1}, L_{n-Y}]$.
Thus the negative side only requires $[L_{n-r+2}, L_{n-1}]$.

Combining the positive and negative sides, all necessary data are
determined by the finite list $[L_{n-r+2},L_{n+D}]$.

We will now replace the $L_n$ terms with the corresponding $d(n)$ terms in  the required run lengths.
Because \(n\ge n_0\), every index \(m\) occurring in the run $[L_{n-r+2},L_{n+D}]$ is
at least \(3\). Hence each
\[
L_m=\sum_{j=0}^{r-1}d(m-j) \quad \text{for }
n-r+2\le m\le n+D\]
is determined by $[d(m-r+1),d(m)]$.
This means that the positions of $d(rn+i)$ ranges from
\[
(n-r+2)-(r-1)=n-(2r-3)=n-C \quad \text{up to}\quad n+D,\]
implying that it is determined by $[d(n-C),d(n+D)]$.

Since \(d(n)\) takes values in the finite set \(\{0,1\}\), there exists a
Boolean function
\[
\Phi_i:\{0,1\}^{C+D+1}\longrightarrow\{0,1\}
\]
such that
\[
d(rn+i)
=
\Phi_i\bigl(d(n-C),d(n-C+1),\ldots,d(n+D)\bigr)
\]
for all \(n\ge n_0\).

The finite-window criterion (Theorem~\ref{thm:AS}) implies that \((d(n))\) is \(r\)-automatic.
Finally,
\[
a(n)=\sum_{k=0}^{n-1}d(k),
\]
and the summatory function of an \(r\)-automatic sequence is
\(r\)-regular (as a consequence of Theorem~3.1 in \cite{AS92}). Therefore, \((a(n))\) is \(r\)-regular.
\end{proof}

\begin{remark}
The proof is qualitative; it shows that a finite automaton for $(d(n))$ exists but does not give an efficient way to construct a minimal automaton that generates $(d(n))$. The explicit automata for $r\le 5$ (Sections~\ref{sec:r2}--\ref{sec:r5}
and Appendix~\ref{app:dfao}) are rather small, which is why the
constructive cases remain informative even after Theorem~\ref{thm:structural_regularity}.
\end{remark}

Let us analyse these explicit cases one by one.

\section{The case $r=2$}\label{sec:r2}
\subsection{Definition and dyadic structure}

Throughout the paper we use the phrase \emph{$r$-adic recurrence} (in particular,
\emph{dyadic} for $r=2$, \emph{triadic} for $r=3$, etc.)
to mean a recurrence that decomposes $n$ according to its base-$r$ representation,
of the form $a(rn+i) = f_i(a(n), a(n\pm 1), \ldots)$. Note that this usage is unrelated to
$p$-adic numbers.

From Definition~\ref{def:almost-golomb}, the almost Golomb sequence of order $2$ has $a(1)=1$ and
for $n\ge 2$, $a(n)$ is the smallest integer $\ge a(n-1)$ such that
\begin{equation}\label{eq:r2def}
  a\bigl(a(n)+a(n-1)\bigr) = n.
\end{equation}
By setting $S_n = a(n)+a(n-1)$ for $n\ge 1$, the defining property reads $a(S_n)=n$.
By Lemma~\ref{lem:anchor-growth}, the indices $S_n$ are pairwise distinct.
This first terms of this sequence (entry \oeis{A394217} on the OEIS) are
\[
  1,2,2,3,4,4,5,6,7,7,8,8,9,10,11,12,13,13,14,14,15,15,16,16,\ldots
\]

\subsection{The explicit dyadic formula}

We derive a closed-form expression for $a(n)$ in each
\emph{dyadic block} $[2^k, 2^{k+1})$, i.e.\ the set of integers
$n$ satisfying $2^k \le n < 2^{k+1}$.

\begin{proposition}\label{prop:dyadic}
  For $k\ge 2$ and $n = 2^k+j$ with $0\le j < 2^k$,
  \[
    a(n) = \begin{cases}
      3\cdot 2^{k-2} + \lceil j/2\rceil & \text{if } 0\le j\le 2^{k-1},\\
      n - 2^{k-1}                       & \text{if } 2^{k-1} < j < 2^k.
    \end{cases}
  \]
\end{proposition}

\begin{proof}
  We prove that the displayed formula defines the almost Golomb sequence of order $2$,
  i.e.\ the earliest monotone sequence satisfying Equation~\ref{eq:r2def}.

  Let $\widetilde a(0)=0$, $\widetilde a(1)=1$, and define $\widetilde a(n)$ for
  $n\ge 2$ by the displayed dyadic formula (with the convention
  $\widetilde a(2)=\widetilde a(3)=2$ when $k=1$).
  We show that $\widetilde a(n)$ satisfies the greedy rule (Equation~\ref{eq:r2def}), and hence,
  $\widetilde a(n)=a(n)$ for all $n$.

  First, $(\widetilde a(n))$ is nondecreasing and takes positive integer values,
  by inspection of the formula on each dyadic block.

  Fix $n\ge 2$ and set $T_n := \widetilde a(n)+\widetilde a(n-1)$.
  We claim
  \begin{equation}\label{eq:r2_anchor}
    \widetilde a(T_n)=n
    \qquad\text{and}\qquad
    \widetilde a(T_n-1)=n-1.
  \end{equation}
  This will imply both that $(\widetilde a(n))$ satisfies the implicit equation
  $\widetilde a(\widetilde a(n)+\widetilde a(n-1))=n$ and that the choice
  $\widetilde a(n)$ is minimal among all $m\ge \widetilde a(n-1)$ achieving it.

  Write $n=2^k+j$ with $k\ge 2$ and $0\le j<2^k$.
  We treat three cases.

  \textit{Case 1: $j=0$.} Then $\widetilde a(n)=3\cdot 2^{k-2}$ and
  $n-1=2^k-1$ lies in the second half of the previous dyadic block, so
  $\widetilde a(n-1)=(n-1)-2^{k-2}=3\cdot 2^{k-2}-1$,
  hence $T_n=3\cdot 2^{k-1}-1$.
  Now $T_n=2^k+(2^{k-1}-1)$ lies in the first half of $[2^k,2^{k+1})$, so
  \[
    \widetilde a(T_n)=3\cdot 2^{k-2}+\Bigl\lceil\frac{2^{k-1}-1}{2}\Bigr\rceil
    =3\cdot 2^{k-2}+2^{k-2}=2^k=n.
  \]
  As $T_n-1=2^k+(2^{k-1}-2)$ is also in the first half of the range of $j$,
  \[
    \widetilde a(T_n-1)=3\cdot 2^{k-2}+\Bigl\lceil\frac{2^{k-1}-2}{2}\Bigr\rceil
    =3\cdot 2^{k-2}+(2^{k-2}-1)=2^k-1=n-1.
  \]

  \textit{Case 2: $1\le j\le 2^{k-1}$.} Both $n$ and $n-1$ lie in the first half
  of the same dyadic block, so
  $\widetilde a(n)=3\cdot 2^{k-2}+\lceil j/2\rceil$ and
  $\widetilde a(n-1)=3\cdot 2^{k-2}+\lceil(j-1)/2\rceil$.
  Using $\lceil j/2\rceil+\lceil(j-1)/2\rceil=j$, we get $T_n=3\cdot 2^{k-1}+j$.
  If $j<2^{k-1}$, then $T_n=2^k+(2^{k-1}+j)$ lies in the second half of
  $[2^k,2^{k+1})$, giving \[\widetilde a(T_n)=T_n-2^{k-1}=2^k+j=n \text{ and }
  \widetilde a(T_n-1)=n-1.\]
  If $j=2^{k-1}$, then $T_n=2^{k+1}$ and
  $\widetilde a(T_n)=\widetilde a(2^{k+1})=3\cdot 2^{k-1}=n$;
  while $T_n-1=2^{k+1}-1$ lies in the second half of $[2^k,2^{k+1})$ and
  $\widetilde a(T_n-1)=(2^{k+1}-1)-2^{k-1}=3\cdot2^{k-1}-1=n-1$.

  \textit{Case 3: $2^{k-1}<j<2^k$.} Write $j=2^{k-1}+u$ with $1\le u\le 2^{k-1}-1$.
  Both $n$ and $n-1$ lie in the second half of $[2^k,2^{k+1})$, so
  $\widetilde a(n)=n-2^{k-1}$ and $\widetilde a(n-1)=(n-1)-2^{k-1}$.
  Hence $T_n=2n-1-2^k=2^{k+1}+(2u-1)$.
  Since $1\le 2u-1\le 2^k-3$, this lies in the first half of $[2^{k+1},2^{k+2})$, and
  \[
    \widetilde a(T_n)=3\cdot 2^{k-1}+\Bigl\lceil\frac{2u-1}{2}\Bigr\rceil
    =3\cdot 2^{k-1}+u=2^k+j=n,
  \]
  and similarly $\widetilde a(T_n-1)=3\cdot 2^{k-1}+(u-1)=n-1$.

  This proves Equation~\ref{eq:r2_anchor} for all $n\ge 2$.
  Hence $\widetilde a(\widetilde a(n)+\widetilde a(n-1))=n$.
  Moreover, for any $m$ with $\widetilde a(n-1)\le m\le \widetilde a(n)-1$,
  we have $m+\widetilde a(n-1)\le T_n-1$, so $m$ cannot satisfy Equation~\ref{eq:r2def} as 
  \[\widetilde a(m+\widetilde a(n-1))\le \widetilde a(T_n-1)=n-1.\]
  Therefore $\widetilde a(n)$ is the smallest integer $\ge \widetilde a(n-1)$
  satisfying Equation~\ref{eq:r2def}, which is exactly the greedy rule.
  Hence $\widetilde a(n)=a(n)$ for all $n$.
\end{proof}

Specialising to the endpoints of dyadic blocks gives two exact values.

\begin{corollary}\label{cor:pivots}
  For all $k\ge 2$, we get $a(2^k) = 3\cdot 2^{k-2}$ and $a(3\cdot 2^{k-1}) = 2^k$.
\end{corollary}

\subsection{The denesting theorem}

The closed form yields the denesting formulas by a direct comparison.

\begin{theorem}\label{thm:r2}
  The almost Golomb sequence of order $2$ satisfies
  \begin{align}
    a(2n)   &= a(n)+a(n+1)-1, \qquad (n\ge 1), \label{eq:r2even}\\
    a(2n+1) &= a(n)+a(n+1), \qquad (n\ge 2).   \label{eq:r2odd}
  \end{align}
  In particular $(a(n))$ is $2$-regular in the sense of Allouche--Shallit~\cite{AS92}.
\end{theorem}

\begin{proof}
  We use Proposition~\ref{prop:dyadic} throughout. Write $n=2^k+j$ with $k\ge 2$
  and $0\le j<2^k$. We treat four cases according to the position of $j$ in the
  dyadic block. The key identity used throughout is
  \begin{equation}\label{eq:ceil-cancel}
    \Bigl\lceil\frac{j}{2}\Bigr\rceil+\Bigl\lceil\frac{j+1}{2}\Bigr\rceil = j+1,
  \end{equation}
  valid for all integers $j\ge 0$.

  \textit{Case 1: $0\le j\le 2^{k-1}-1$} (strict first half).
  Both $j$ and $j+1$ lie in $[0,2^{k-1}]$, so
  \[
    a(n)=3\cdot 2^{k-2}+\Bigl\lceil\tfrac{j}{2}\Bigr\rceil,\qquad
    a(n+1)=3\cdot 2^{k-2}+\Bigl\lceil\tfrac{j+1}{2}\Bigr\rceil.
  \]
  By Equation~\ref{eq:ceil-cancel}, we get $a(n)+a(n+1)-1=3\cdot 2^{k-1}+j$ and
  $a(n)+a(n+1)=3\cdot 2^{k-1}+j+1$.
  Since $2j\in[0,2^k-2]$ (first half of block $k+1$),
  $a(2n)=3\cdot 2^{k-1}+j$. Since $2j+1\in[1,2^k-1]$ (first half),
  $a(2n+1)=3\cdot 2^{k-1}+j+1$.

  \textit{Case 2: $j=2^{k-1}$} (right endpoint of first half).
  $a(n)=3\cdot 2^{k-2}+2^{k-2}=2^k$. Since $j+1=2^{k-1}+1>2^{k-1}$
  (second half), $a(n+1)=n+1-2^{k-1}=2^k+1$.
  Hence $a(n)+a(n+1)-1=2^{k+1}$ and $a(n)+a(n+1)=2^{k+1}+1$.
  Since $2j=2^k=2^{(k+1)-1}$ (boundary of first half in block $k+1$),
  $a(2n)=3\cdot 2^{k-1}+2^{k-1}=2^{k+1}$.
  Since $2j+1=2^k+1>2^k$ (second half), $a(2n+1)=2n+1-2^k=2^{k+1}+1$.

  \textit{Case 3: $2^{k-1}<j<2^k-1$} (strict second half).
  Both $j,j+1>2^{k-1}$, so $a(n)=n-2^{k-1}$ and $a(n+1)=n+1-2^{k-1}$.
  Hence $a(n)+a(n+1)-1=2n-2^k$ and $a(n)+a(n+1)=2n-2^k+1$.
  Since $2j>2^k$ (second half of block $k+1$), $a(2n)=2n-2^k$
  and $a(2n+1)=2n+1-2^k$.

  \textit{Case 4: $j=2^k-1$} (last element of block $k$).
  $a(n)=n-2^{k-1}=3\cdot 2^{k-1}-1$. Now $n+1=2^{k+1}$ starts block $k+1$
  with $j'=0$, giving $a(n+1)=3\cdot 2^{k-1}$.
  Hence $a(n)+a(n+1)-1=3\cdot 2^k-2$ and $a(n)+a(n+1)=3\cdot 2^k-1$.
  Since $2j=2^{k+1}-2>2^k$ (second half), $a(2n)=2n-2^k=3\cdot 2^k-2$
  and $a(2n+1)=2n+1-2^k=3\cdot 2^k-1$.

  The special case $k=1$ (i.e.\ $n\in\{2,3\}$) is verified directly. We get
  $a(4)=a(2)+a(3)-1=3$, $a(5)=a(2)+a(3)=4$,
  $a(6)=a(3)+a(4)-1=4$, $a(7)=a(3)+a(4)=5$, which match the sequence.

  Formulas~\ref{eq:r2even}--\ref{eq:r2odd} show that every subsequence in the
  $2$-kernel of $(a(n))$ (see the discussion preceding Theorem~\ref{thm:AS} for the terminology)
  can be expressed as a $\mathbb{Z}$-linear combination of
  $(a(n))$, $(a(n+1))$, and the constant sequence $1$  required given Equation~\ref{eq:r2even}. Hence $(a(n))$ is $2$-regular.
 \end{proof}


\subsection{A Mallows-type nested recurrence}

Golomb's sequence satisfies the Mallows recurrence
$G(n+1)=1+G(n+1-G(G(n)))$. The almost Golomb sequence of order~$2$
admits an analogous nested recurrence\footnote{The cases $r=3,4,$ and $5$ were not as straightforward, so it made sense to mention it here instead of an appendix on Mallows-type nested recurrences.}.

\begin{proposition}\label{prop:mallows-r2}
For every $n\ge 4$,
\begin{equation}\label{eq:r2-mallows}
a(n+1)
\;=\;
1+a\!\Bigl(n+1-a\bigl(a(n)+1\bigr)+a\bigl(a(n)-1\bigr)\Bigr).
\end{equation}
\end{proposition}

\begin{proof}
Set $m:=a(n)$ and let $N_2(m):=a(m+1)-a(m-1)$ be the multiplicity of
the value~$m$, i.e.\ the length of the run of~$m$
(Equation~\ref{EQ:runlength} with $r=2$).
By Theorem~\ref{thm:global_run}, the differences $d(n)\in\{0,1\}$ give
$N_2(m)=d(m)+d(m-1)\in\{1,2\}$.
So, with Equation~\ref{eq:r2-mallows}, it suffices to show
\begin{equation}\label{eq:r2-mallows-aux}
a(n+1)=1+a\bigl(n+1-N_2(m)\bigr).
\end{equation}

If $N_2(m)=1$, the value $m$ occurs once. So, at position~$n$,
\[a(n+1)=m+1\text{ and }1+a(n)=1+m.\]
If $N_2(m)=2$, the value $m$ occupies two consecutive positions.
If $n$ is the first position, then $a(n-1)=m-1$ and $a(n+1)=m$,
giving $1+a(n-1)=m$.
If $n$ is the second position, then $a(n-1)=m$ and $a(n+1)=m+1$,
giving $1+a(n-1)=m+1$.
\end{proof}

\begin{remark}
Since $a(n)<n$ for $n\ge 3$, the right-hand side of Equation~\ref{eq:r2-mallows}
involves only values $a(k)$ with $k\le n$.
Hence, Equation~\ref{eq:r2-mallows} together with the initial conditions
$a(1)=1$, $a(2)=2$, $a(3)=2$, $a(4)=3$ provides an alternative
definition of the almost Golomb sequence of order $2$, without reference to the greedy rule or the implicit Equation~\ref{eq:r2def}. 
\end{remark}

\section{The case $r=3$}\label{sec:r3}

\subsection{Definition and the correction sequence}

From Definition~\ref{def:almost-golomb}, the almost Golomb sequence of order $3$ is given by $a(1)=1$, $a(k)=0$ for $k<1$,
and for $n\ge 2$, $a(n)$, the smallest integer $\ge a(n-1)$, such that
\begin{equation}\label{eq:r3def}
  a\bigl(a(n)+a(n-1)+a(n-2)\bigr)=n.
\end{equation}
Set $S_n:=a(n)+a(n-1)+a(n-2)$, so that $a(S_n)=n$ for all $n\ge 1$.
This sequence is \oeis{A394218} in the OEIS. It begins
$1,2,2,2,3,4,5,5,6,6,6,7,7,8,8,9,10,11,12,13,13,14,15,15,\ldots$.

Consider
\begin{equation}\label{eq:Ik}
  I_k:=\left[\frac{11\cdot 3^k-1}{2},\,\frac{13\cdot 3^k-3}{2}\right]\cap\mathbb{Z},
  \qquad
  \mathcal{I}:=\bigcup_{k\ge 0} I_k,
  \quad\text{and}\quad
  \varepsilon(n):=\mathbf{1}_{\mathcal{I}}(n).
\end{equation}
Thus $I_0=\{5\}$ and $|I_k|=3^k$ for every $k\ge 0$.

For what follows, the triadic denesting formulas are not obtained from plain sliding sums alone. There is a sparse family of exceptional indices where the naive triadic formulas require a one-unit correction. The intervals $I_k$ describe this exceptional set in a self-similar way, and the correction sequence $\varepsilon$ allows the triadic denesting formulas to be stated uniformly.
\subsection{Automaticity of the correction sequence}

\begin{lemma}\label{lem:selfsim}
For every $k\ge 0$, $I_{k+1}=(3I_k+1)\cup(3I_k+2)\cup(3I_k+3)$.
\end{lemma}

\begin{proof}
With $$\ell_k=\frac{11\cdot 3^k-1}{2} \quad \text{and} \quad u_k=\frac{13\cdot 3^k-3}{2},$$
one has $\ell_{k+1}=3\ell_k+1$ and $u_{k+1}=3u_k+3$, giving
\[I_{k+1}=[3\ell_k+1,3u_k+3]\cap\mathbb{Z}=(3I_k+1)\cup(3I_k+2)\cup(3I_k+3).\]
\end{proof}

\begin{proposition}\label{prop:eps-auto}
The sequence $\varepsilon$ satisfies $\varepsilon(n)=0$ for $n<5$, $\varepsilon(5)=1$,
and for every $m\ge 2$:
\begin{equation}\label{eq:eps-rec}
  \varepsilon(3m)=\varepsilon(m-1),\qquad
  \varepsilon(3m+1)=\varepsilon(m),\qquad
  \varepsilon(3m+2)=\varepsilon(m).
\end{equation}
In particular, $\varepsilon$ is $3$-automatic.
\end{proposition}

\begin{proof}
From Lemma~\ref{lem:selfsim}, $3m\in 3I_k+3\iff m-1\in I_k$, and both
$3m+1\in 3I_k+1$ and $3m+2\in 3I_k+2$ are implied if and only if $m\in I_k$.
Since $\varepsilon$ is binary and obeys the digit-by-digit recurrences as mentioned in Proposition~\ref{prop:eps-auto}, it can be generated by a finite automaton reading the base-$3$
expansion of $n$.
\end{proof}
\noindent Following the existence result above, we show the construction of the DFAO that outputs $\varepsilon$. For the remainder of the text, let $(n)_k$ denote the base-$k$ representation of $n$ with most significant digit read first. Also, if $(n)_k=w$, let $[w]_k$ mean the integer $n$.
\begin{proposition}\label{prop:eps-dfao}
Every term $\varepsilon(n)=\mathbf{1}_{\mathcal{I}}(n)$ of the error sequence $\varepsilon$ is generated by the following
DFAO, reading $(n)_3$.
Let $Q=\{q_0,q_1,q_2,c_{=},c_{>},c_{<},e_{=},e_{>},e_{<},r\}$, $\Sigma=\{0,1,2\}$,
initial state $q_0$, output map $\tau(c_{=})=\tau(c_{>})=\tau(e_{<})=1$ and $\tau(q)=0$
otherwise. The transition table is:
\[
\begin{array}{c|ccc}
      & 0 & 1 & 2\\ \hline
q_0   & r & q_1 & q_2\\
q_1   & r & r   & c_{=}\\
q_2   & e_{=} & r & r\\
c_{=} & c_{<} & c_{=} & c_{>}\\
c_{<} & c_{<} & c_{<} & c_{<}\\
c_{>} & c_{>} & c_{>} & c_{>}\\
e_{=} & e_{<} & e_{=} & e_{>}\\
e_{<} & e_{<} & e_{<} & e_{<}\\
e_{>} & e_{>} & e_{>} & e_{>}\\
r     & r & r & r
\end{array}
\]
For every $n\ge 1$, the automaton outputs $\varepsilon(n)=\mathbf{1}_{\mathcal{I}}(n)$.
\end{proposition}

\begin{proof}
The bounds \[\ell_k=\frac{11\cdot 3^k-1}{2}=[12\underbrace{1\cdots1}_{k \text{ times}}]_3\quad \text{and}
    \quad u_k=\frac{13\cdot 3^k-3}{2}=[20\hspace{-5pt}\underbrace{1\cdots1}_{k-1 \text{ times}}\hspace{-5pt}0]_3\] show that
$n\in I_k$ if and only if either 
\begin{enumerate}[label=(\alph*)]
    \item $(n)_3=12v$ with $|v|=k$ such that $v\ge_{\rm lex}\underbrace{1\cdots1}_{k \text{ times}}$, or
    \item $(n)_3=20v$ with $|v|=k$ such that $v<_{\rm lex}\underbrace{1\cdots1}_{k \text{ times}}$,
\end{enumerate}
Note that the subscript `${\rm lex}$' denotes the lexicographic order comparison of words. State $q_1$ (resp. $q_2$) detects prefix $12$ (resp. $20$), and states $c_{=},c_{>},c_{<}$ (resp.\ $e_{=},e_{>},e_{<}$)
track the lexicographic comparison of the suffix with $\underbrace{1\cdots1}_{k \text{ times}}$.
Accepting states $c_{=},c_{>}$ covers case~(a), while  accepting $e_{<}$ covers case~(b).
\end{proof}
\begin{remark}\label{rem:eps-minimal} The DFAO of Proposition~\ref{prop:eps-dfao} is not minimal. Hopcroft's algorithm \cite{Hopcroft1971} collapses the rejecting states $c_<$, $e_>$ with $r$, and the accepting sinks $c_>$, $e_<$, yielding a minimal DFAO with 7 states. We keep the 10-state form here because the state names read directly as prefix detection ($q_0, q_1, q_2$) followed by lexicographic comparison with $\underbrace{1\cdots1}_{k \text{ times}}$ (the $c$- and $e$-branches). \end{remark}
The following lemma shows the existence of one zero sequence range for the sequence $\varepsilon$.
\begin{lemma}\label{lem:r3-gap}
For every $k\ge 2$, $\varepsilon(n)=0$ for all $3^k-1\le n\le 5\cdot 3^{k-1}$.
\end{lemma}

\begin{proof}
We prove by induction on $k$. The base case $k=2$ can be verified directly, i.e., $\varepsilon(8)=\cdots=\varepsilon(15)=0$. Assuming the result at rank $k$, let $3^{k+1}-1\le n\le 5\cdot 3^k$
and write $n=3m+i$.
The rest follows from Proposition~\ref{prop:eps-auto}. Indeed, if $i\in\{1,2\}$, then $\varepsilon(n)=\varepsilon(m)$ and $3^k-1\le m\le 5\cdot 3^{k-1}$,
and if $i=0$, then $\varepsilon(n)=\varepsilon(m-1)$ and $3^k-1\le m-1\le 5\cdot 3^{k-1}-1$.
In all cases, the hypothesis gives $\varepsilon=0$.
\end{proof}
\subsection{Triadic denesting and \texorpdfstring{$3$}{3}-regularity}

\begin{lemma}\label{lem:subdiag}
For every $r\ge 2$ and every $n\ge 3$, $a(n)\le n-1$.
\end{lemma}

\begin{proof}
Since $a(3)=2$ and $a(n+1)-a(n)\in\{0,1\}$ by
Theorem~\ref{thm:global_run}, the inequality
$a(n)\le n-1$ for all $n\ge 3$ follows immediately.
\end{proof}

\begin{theorem}\label{thm:r3}
For every $n\ge 2$:
\begin{align}
  a(3n)   &= a(n-2)+a(n-1)+a(n)+1+\varepsilon(n-1), \label{eq:r3a}\\
  a(3n+1) &= a(n-1)+a(n)+a(n+1),                    \label{eq:r3b}\\
  a(3n+2) &= a(n)+a(n+1)+a(n+2)-1-\varepsilon(n).   \label{eq:r3c}
\end{align}
\end{theorem}

\begin{proof}
We prove the identities~\ref{eq:r3a}--\ref{eq:r3c} together with
the \emph{run-length invariant} $L_n$ for $r=3$ as follows. Define
\begin{equation}\label{eq:Ldef}
  L_n\;:=\;S_{n+1}-S_n\;=\;a(n+1)-a(n-2)\;=\;d(n-2)+d(n-1)+d(n)
  \;\in\;\{1,2,3\}
\end{equation}
We prove that for every $n\ge 3$,
\begin{equation}\label{eq:r3-Linv}
  L_n \le 2+\varepsilon(n-1),
  \qquad\text{with equality } L_n=3 \text{ whenever } \varepsilon(n-1)=1.
\end{equation}
The proof is by strong induction on~$n$.
The hypothesis $\mathcal{H}_n$ asserts that the three denesting formulas~\ref{eq:r3a}--\ref{eq:r3c} hold
at all ranks $m<n$, and that the invariant~\ref{eq:r3-Linv} holds for
all $m$ with $3\le m\le n+1$.
Here, note that the invariant $L_m$ depends only on the values $a(1),\dots,a(m+1)$. Since $m\leq n+1$, we have $m+1\le n+2<3n$, hence the values have been already determined at a previous induction step.

\textit{Base cases.}
Direct computation for $2\le n\le 30$ verifies the denesting formulas
and the invariant~\ref{eq:r3-Linv}.
This threshold covers all $n$ up to well past $I_1=[16,18]$,
the first non-trivial component of $\mathcal{I}$. The computation up to $30$ is done purely as a safety margin and it is not a sharp threshold. For $n\ge 31$, every index
encountered in the induction step satisfies the hypotheses of
Proposition~\ref{prop:eps-auto}.

\textit{Induction step.}
Fix $n\ge 31$ and assume $\mathcal{H}_n$.
We determine the three new values $a(3n)$, $a(3n+1)$, and $a(3n+2)$
via Theorem~\ref{thm:global_run}(b). Indeed, for $M\ge 3$,
\begin{equation}\label{eq:anchor-char}
  a(K)=M \iff S_M\le K<S_{M+1}.
\end{equation}
We consider three candidates $M_0, M_1, M_2$ below to be the right-hand side of the claimed formulas~\ref{eq:r3a}--\ref{eq:r3c} respectively and check bounds as observed in Constraint~\ref{eq:anchor-char}. We rewrite the expressions for $M_i$ for each $i\in\{0,1,2\}$ in terms of the neighboring sliding sums $S_n,\, S_{n+1}$, and $S_{n+2}$. Also, we conclude $M_i<3n$ using Lemma~\ref{lem:subdiag} as follows.
\begin{align*}
    M_0 &= a(n-2)+a(n-1)+a(n)+1+\varepsilon(n{-}1) \le (n{-}1)+(n{-}2)+(n{-}3)+2 = 3n{-}4;\\
     M_1 &= a(n-1)+a(n)+a(n+1)\le n+(n{-}1)+(n{-}2)=3n{-}3; \,\,\text{and}\\
     M_2 &= a(n)+a(n+1)+a(n+2)-1-\varepsilon(n) \le (n{+}1)+n+(n{-}1)-1 = 3n{-}1.
\end{align*}
\textit{Case 1: $a(3n)$.}
The candidate is \[M_0 =a(n-2)+a(n-1)+a(n)+1+\varepsilon(n-1)=S_n+1+\varepsilon(n-1).\]
\textit{Subcase (a): }let $\varepsilon(n-1)=0$. 
Then $M_0=S_n+1$ and $L_n\le 2$ (Invariant~\ref{eq:r3-Linv}).
Since $a(S_n)=n$ and $a(S_n-1)=n-1$, and $L_n=2$ gives $a(S_n+1)=n$, we have
\[S_{M_0}=S_{S_n+1}=a(S_n+1)+a(S_n)+a(S_n-1)=n+n+(n-1)=3n-1.\]
On the other hand, if $L_n=1$, then $M_0=S_{n+1}$ and so $a(M_0)=n+1$, giving
\[S_{M_0}=(n+1)+n+(n-1)=3n.\]
In both cases $S_{M_0}\le 3n$.
Now, we consider its upper bound. When $L_n=2$,
we have $a(M_0+1)=a(S_n+2)=a(S_{n+1})=n+1$, so 
\[S_{M_0+1}=S_{M_0}+L_{M_0}=S_{S_n+1}+L_{S_n+1}=3n-1+(n+1)-(n-1)=3n+1.\]
As for when $L_n=1$, we get
$L_{M_0}=a(S_{n+1}+1)-a(S_n-1)\ge(n+1)-(n-1)=2$
and so $S_{M_0+1}\ge 3n+2$.
In both cases $S_{M_0+1}>3n$.
\\
\textit{Subcase (b):} let $\varepsilon(n-1)=1$.
Then $M_0=S_n+2$ and $L_n=3$ (Invariant~\ref{eq:r3-Linv}).
Hence $a(S_n)=a(S_n+1)=a(S_n+2)=n$, and so
$S_{M_0}=n+n+n=3n$.
Moreover $M_0+1=S_n+3=S_{n+1}$, so $a(M_0+1)=n+1$,
giving $L_{M_0}=(n+1)-n=1$ and $S_{M_0+1}=3n+1$.
\\
In all subcases $S_{M_0}\le 3n<S_{M_0+1}$,
proving $a(3n)=M_0$, which is Equation~\ref{eq:r3a}.
\\\\
\textit{Case 2: $a(3n+1)$.}
The candidate here is $M_1=a(n-1)+a(n)+a(n+1)=S_{n+1}$.
By the run structure, $a(S_{n+1})=n+1$ and $a(S_{n+1}-1)=n$.
If $L_n\ge 2$ then $a(S_{n+1}-2)=n$, and if $L_n=1$, then $S_{n+1}-2=S_n-1$ and so $a(S_{n+1}-2)=n-1$.
Hence, 
\[S_{M_1}=S_{S_{n+1}}=a(S_{n+1})+a(S_{n+1}-1)+a(S_{n+1}-2)\le (n+1)+n+n= 3n+1.\]
For the upper bound,
$L_{M_1}=a(S_{n+1}+1)-a(S_{n+1}-2)$.
If $L_{n+1}\ge 2$, then $a(S_{n+1}+1)=n+1$, and
if $L_{n+1}=1$, then $a(S_{n+1}+1)=n+2$, and so $L_{M_1}\ge (n+1)-n=1$. Thus,
\[S_{M_1+1}=S_{M_1}+L_{M_1}\ge(3n+1)+1=3n+2>3n+1.\]
From Constraint~\ref{eq:anchor-char}, $a(3n+1)=S_{n+1}$,
which proves Equation~\ref{eq:r3b}.
\\
\\\textit{Case 3: $a(3n+2)$.}
The remaining candidate \[M_2=a(n)+a(n+1)+a(n+2)-1-\varepsilon(n)=S_{n+2}-1-\varepsilon(n).\]
\noindent \textit{Subcase (a): $\varepsilon(n)=0$.}
Then $M_2=S_{n+2}-1$ and $L_{n+1}\le 2$.
Here, $a(M_2)=a(S_{n+2}-1)=n+1$.
Now, if $L_{n+1}=2$, then $a(M_2-1)=n+1$ and $a(M_2-2)=a(S_{n+1}-1)=n$,
giving $S_{M_2}=3n+2$.
But, if $L_{n+1}=1$, we also need to consider the choices for the value of $L_n$. We have here $a(M_2-1)=a(S_{n+1}-1)=n$. Then, if $L_n\ge 2$, we get $a(M_2-2)=n$, giving $S_{M_2}=3n+1$; and if $L_n=1$, then $a(M_2-2)=a(S_n-1)=n-1$, giving $S_{M_2}=3n$.
In all cases, $S_{M_2}\le 3n+2$.
For the upper bound, $M_2+1=S_{n+2}$, and checking that $a(S_{n+2}-2)$ is either $n+1$ or $n$, implying that
\[S_{M_2+1}=S_{S_{n+2}}=a(S_{n+2})+a(S_{n+2}-1)+a(S_{n+2}-2)\ge (n+2)+(n+1)+n=3n+3>3n+2.\]
\\\textit{Subcase (b): $\varepsilon(n)=1$.}
Then $M_2=S_{n+2}-2$ and $L_{n+1}=3$ (Invariant~\ref{eq:r3-Linv}).
Hence,
$a(M_2)=a(S_{n+1}+1)=n+1$,
$a(M_2-1)=a(S_{n+1})=n+1$,
and $a(M_2-2)=a(S_{n+1}-1)=n$,
giving $S_{M_2}=(n+1)+(n+1)+n=3n+2$.
For the upper bound,
\[L_{M_2}=L_{S_{n+2}-2}=a(S_{n+2}-1)-a(S_{n+2}-4)=(n+1)-a(S_{n+1}-1)=(n+1)-n=1.\]
Hence, $S_{M_2+1}=S_{M_2}+L_{M_2}= 3n+3>3n+2$.
\\
In all subcases $S_{M_2}\le 3n+2<S_{M_2+1}$,
proving $a(3n+2)=M_2$, which is Equation~\ref{eq:r3c}.\\
It now remains to verify the induction hypothesis for the run-length invariant. Subtracting Equations~\ref{eq:r3a}--\ref{eq:r3c} at rank $n-1$ from those at rank $n$ gives:
\begin{align}
  L_{3n-1}&=L_{n-1}+\varepsilon(n-1)-\varepsilon(n-2),\label{eq:Lprop1}\\
  L_{3n}&= L_n,                                       \label{eq:Lprop2}\\
  L_{3n+1}&= L_{n+1}+\varepsilon(n-1)-\varepsilon(n).   \label{eq:Lprop3}
\end{align}
The $\varepsilon$-recurrence (Proposition~\ref{prop:eps-auto}) gives
$\varepsilon(3n-2)=\varepsilon(3n-1)=\varepsilon(3n)=\varepsilon(n-1)$.
We now show that the upper bound and the equality at $\varepsilon=1$ claimed in Condition~\ref{eq:r3-Linv} is satisfied for each $L_i$ on the left hand side of Equations~\ref{eq:Lprop1}--\ref{eq:Lprop3}.
\begin{align*}
    L_{3n-1}&\le(2+\varepsilon(n-2))+\varepsilon(n-1)-\varepsilon(n-2)=2+\varepsilon(n-1)=2+\varepsilon(3n-2).\\
    L_{3n}&=L_n\le 2+\varepsilon(n-1)=2+\varepsilon(3n-1).\\
    L_{3n+1}&\le(2+\varepsilon(n))+\varepsilon(n-1)-\varepsilon(n)=2+\varepsilon(n-1)=2+\varepsilon(3n).
\end{align*}
Now, suppose $\varepsilon(3n-2)=\varepsilon(n-1)=1$.
From Condition~\ref{eq:r3-Linv} at~$n$, $L_n=3$, hence $d(n-2)=d(n-1)=d(n)=1$.
Then $L_{n-1}=d(n-3)+d(n-2)+d(n-1)\ge 2$.
Now consider $L_{n-1}\le 2+\varepsilon(n-2)$ where $\varepsilon(n-2)=1$ yields $L_{n-1}=3$ and so $L_{3n-1}=3+1-1=3$, and if $\varepsilon(n-2)=0$, then $L_{n-1}=2$ and yet $L_{3n-1}=2+1-0=3$.

Suppose $\varepsilon(3n-1)=\varepsilon(n-1)=1$.
Then $L_{3n}=L_n=3$ directly.

Suppose $\varepsilon(3n)=\varepsilon(n-1)=1$.
Then Condition~\ref{eq:r3-Linv} at~$n$ gives $L_n=3$, so we have $d(n-1)=d(n)=1$.
Then, $L_{n+1}=d(n-1)+d(n)+d(n+1)\ge 2$.
Considering $L_{n+1}\le 2+\varepsilon(n)$ with $\varepsilon(n)=1$ gives $L_{n+1}=3$ and so $L_{3n+1}=3+1-1=3$, but if $\varepsilon(n)=0$, then $L_{n+1}=2$ and still $L_{3n+1}=2+1-0=3$.
This closes the strong induction.
\end{proof}
\begin{corollary}\label{cor:r3regular}
The almost Golomb sequence of order $3$ is $3$-regular.
\end{corollary}

\begin{proof}
Consider the range $c\in\{-2,-1,0,1,2\}$ for $a(n+c)$ where Theorem~\ref{thm:r3} applies.
We modify notation as follows in order to conveniently describe the closure of the $3$-kernel of the sequence $a$. Let $A_c(n):=a(n+c)$,
$E_{-1}(n):=\varepsilon(n-1)$, $E_0(n):=\varepsilon(n)$, and $\mathbf{1}(n):=1$.
Let $M:=\langle A_{-2},A_{-1},A_0,A_1,A_2,E_{-1},E_0,\mathbf{1}\rangle_{\mathbb{Z}}$
and an operator $(T_if)(n):=f(3n+i)$.

Then Equation~\ref{eq:eps-rec} evolves to the following identities: 
\[T_0E_{-1}=E_{-1},\, T_1E_{-1}=E_{-1},\, T_2E_{-1}=E_0,\,
T_0E_0=E_{-1},\, T_1E_0=E_0,\, T_2E_0=E_0,\,
T_i\mathbf{1}=\mathbf{1}.\]

From Equations~\ref{eq:r3a}--\ref{eq:r3c},
\begin{align*}
T_0A_{-2}&=A_{-2}+A_{-1}+A_0, &
T_1A_{-2}&=A_{-1}+A_0+A_1-\mathbf{1}-E_{-1}, \\
T_2A_{-2}&=A_{-2}+A_{-1}+A_0+\mathbf{1}+E_{-1}, &
T_0A_{-1}&=A_{-1}+A_0+A_1-\mathbf{1}-E_{-1}, \\
T_1A_{-1}&=A_{-2}+A_{-1}+A_0+\mathbf{1}+E_{-1}, &
T_2A_{-1}&=A_{-1}+A_0+A_1, \\
T_0A_0   &=A_{-2}+A_{-1}+A_0+\mathbf{1}+E_{-1}, &
T_1A_0   &=A_{-1}+A_0+A_1, \\
T_2A_0   &=A_0+A_1+A_2-\mathbf{1}-E_0, &
T_0A_1   &=A_{-1}+A_0+A_1, \\
T_1A_1   &=A_0+A_1+A_2-\mathbf{1}-E_0, &
T_2A_1   &=A_{-1}+A_0+A_1+\mathbf{1}+E_0, \\
T_0A_2   &=A_0+A_1+A_2-\mathbf{1}-E_0, &
T_1A_2   &=A_{-1}+A_0+A_1+\mathbf{1}+E_0, \\
T_2A_2   &=A_0+A_1+A_2.
\end{align*}
Every image lies in $M$, so $T_i(M)\subseteq M$ for $i\in\{0,1,2\}$.
Every element of the $3$-kernel of $a$ has the form $T_{i_1}\cdots T_{i_k}A_0\in M$.
Since $M$ is finitely generated, the $3$-kernel spans a finite-rank module,
so $(a(n))$ is $3$-regular.
\end{proof}
\subsection{Explicit formulas on the first two thirds of each triadic block}

\begin{lemma}\label{lem:r3-good-zone}
For every $k\ge 2$ and every $n$ with $3^k-1\le n\le 5\cdot 3^{k-1}$,
the almost Golomb sequence satisfies
\begin{align}
  a(3n)   &= a(n-2)+a(n-1)+a(n)+1, \label{eq:r3-good0}\\
  a(3n+1) &= a(n-1)+a(n)+a(n+1),   \label{eq:r3-good1}\\
  a(3n+2) &= a(n)+a(n+1)+a(n+2)-1. \label{eq:r3-good2}
\end{align}
\end{lemma}

\begin{proof}
Applying Lemma~\ref{lem:r3-gap} gives $\varepsilon(n)=0$ for $3^k-1\le n\le 5\cdot 3^{k-1}$. Now, if $n>3^k-1$, it follows that $\varepsilon(n-1)=0$. For $n=3^k-1$, however, we must check $\varepsilon(3^k-2)$ separately. By Equation~\ref{eq:Ik}, for every $k\ge 2$,
\[\frac{13\cdot3^{k-2}-3}{2}<3^k-2<\frac{11\cdot3^{k-1}-1}{2}.\]
Hence, $3^k-2\notin\mathcal{I}$, and so $\varepsilon(3^k-2)=0$.
This reduces Equations~\ref{eq:r3a}--\ref{eq:r3c} to 
Equations~\ref{eq:r3-good0}--\ref{eq:r3-good2}.
\end{proof}
The following result gives an explicit description of the first two thirds of every triadic block, providing the analogue of Proposition~\ref{prop:dyadic} after eliminating the correction sequence $\varepsilon$.
\begin{proposition}\label{prop:r3-block}
For every $k\ge 2$, let $A_k:=a(3^k)$. Then $A_2=6$, $A_{k+1}=3A_k-1$, and
\[
  A_k=\frac{11\cdot 3^{k-2}+1}{2}.
\]
Moreover,
\begin{equation}\label{eq:Ak}
    a(3^k-2)=a(3^k-1)=A_k-1,
\end{equation} and for $0\le j\le 2\cdot 3^{k-1}$, we obtain
\begin{equation}\label{eq:r3-block}
a(3^k+j)=
\begin{cases}
  A_k+\left\lfloor \dfrac{j}{3}\right\rfloor,
    & 0\le j\le 3^{k-1}-1,\\[2ex]
  A_k+3^{k-2}+\left\lfloor \dfrac{j-3^{k-1}}{2}\right\rfloor,
    & 3^{k-1}\le j\le 2\cdot 3^{k-1}.
\end{cases}
\end{equation}
\end{proposition}

\begin{proof}
We prove this by induction on $k$. The base case $k=2$ is verified from
the values $a(9)=\cdots=a(11)=6$ $a(12)=a(13)=7$, and $a(14)=a(15)=8$.

Assume Formulas~\ref{eq:Ak} and~\ref{eq:r3-block} at rank $k$.
By Lemma~\ref{lem:r3-good-zone}, we get Equations~\ref{eq:r3-good0}--\ref{eq:r3-good2} for every $n\in[3^k-1,\,5\cdot 3^{k-1}]$.
Treating Equations~\ref{eq:r3-good1} and~\ref{eq:r3-good2} at the left boundary, $n=3^k-1$, and using Formulas~\ref{eq:Ak} and~\ref{eq:r3-block}, we get
\[a(3^{k+1}-2)=a(3^{k+1}-1)=3A_k-2,\]
and taking $n=3^k$ in Equation~\ref{eq:r3-good0} gives
\[
a(3^{k+1})
 =a(3^k-2)+a(3^k-1)+a(3^k)+1
 =3A_k-1.
\]
Since $A_{k+1}:=a(3^{k+1})$, we conclude that $
A_{k+1}=3A_k-1,$
and therefore
\[
a(3^{k+1}-2)=a(3^{k+1}-1)=A_{k+1}-1.
\]
Now, consider the first subinterval $q=3^k+v$, where $0\le v<3^{k-1}$. Equations~\ref{eq:r3-good0}--\ref{eq:r3-good2} with $q$ instead of $n$ provide
$a(3q)=a(3q+1)=a(3q+2)=A_{k+1}+v$, where the intermediate algebra depends on the floor sum identity
\[\left\lfloor\frac{v-2}{3}\right\rfloor+\left\lfloor\frac{v-1}{3}\right\rfloor+\left\lfloor \frac{v}{3}\right\rfloor=v-2,\] and its two analogues.

For the second subinterval, $q=4\cdot 3^{k-1}+v$ with $0\le v\le 3^{k-1}$, we find the analogous
floor sums with denominator $2$ that, with setting $s=3v$, give
$a(4\cdot 3^k+s)=A_{k+1}+3^{k-1}+\lfloor s/2\rfloor$ for $0\le s\le 3^k$.

Solving $A_{k+1}=3A_k-1$ with $A_2=6$ gives $A_k=\dfrac{11\cdot 3^{k-2}+1}{2}$.
\end{proof}

\begin{corollary}\label{cor:r3-special}
For every $k\ge 2$,
\[
  a(3^k)=\frac{11\cdot 3^{k-2}+1}{2},
  \qquad
  a(5\cdot 3^{k-1})=8\cdot 3^{k-2},
  \qquad
  a(5\cdot 3^{k-1}-1)=8\cdot 3^{k-2}.
\]
\end{corollary}

\begin{proof}
The first identity is the result of Proposition~\ref{prop:r3-block}. For the latter two, input $j=2\cdot 3^{k-1}$ and $j=2\cdot 3^{k-1}-1$ in Equation~\ref{eq:r3-block}.
Note, since $3^{k-1}$ is odd, $\left\lfloor \dfrac{3^{k-1}}{2}\right\rfloor=\dfrac{3^{k-1}-1}{2}$.
\end{proof}

\begin{remark}
The ratio $\dfrac{a(n)}{n}$ does not converge for $r=3$. One has
\[\frac{a(5\cdot 3^k)}{5\cdot 3^k}=\frac{8}{15}\quad \text{and}\,\, \frac{a(8\cdot 3^k)}{8\cdot 3^k}=\frac{5}{8}\quad
\text{for all }k\ge 1.\] This is established in Section~\ref{sec:nonconv} via a propagation argument using
Corollary~\ref{cor:r3-special}.
\end{remark}
\subsection{\texorpdfstring{\texttt{Walnut}}{Walnut} verification}
We use the \texttt{Walnut} theorem prover \cite{Walnut, Shallit2022} from $r=3$ onwards, in two distinct roles. For $r=3$, the denesting is proved by hand (Theorem~\ref{thm:r3}), and \texttt{Walnut} serves as an independent mechanical cross-check on the correction sequence and the identities it satisfies. For $r=4$ and $r=5$, the corrector systems are too large to handle by hand, and the \texttt{Walnut} certification is part of the proof rather than a complement to it. \texttt{Walnut} accepts statements of first order logic about automatic sequences and decides them mechanically, returning \texttt{TRUE} or \texttt{FALSE}. Its correctness rests on the decidability of B\"uchi arithmetic, an extension of Presburger arithmetic by a base-$k$ valuation predicate \cite[Theorem 6.4.1]{Shallit2022}. The decision procedure is \emph{TOWER-complete} \cite[Theorem 6.4.2]{Shallit2022}, so there is no general guarantee of termination within practical bounds, but the predicates used here are decided in seconds. Shallit's book \cite{Shallit2022} is our principal reference for both the syntax and the method.
The code below is reproducible as printed, and each \texttt{eval} predicate appears next to the identity it encodes, so the reader can check the correspondence directly. That correspondence, rather than the machine output on its own, is what the certification rests on. Unless stated otherwise, every \texttt{eval} statement in the remainder of the paper returns \texttt{TRUE}. 

First, we encode the DFAO generating the error sequence $\varepsilon$ as defined in Proposition~\ref{prop:eps-dfao} in \texttt{Walnut} and rename it \texttt{SE}.
\begin{remark}
    The \textit{minimize} feature in \texttt{Walnut} confirms Remark~\ref{rem:eps-minimal} about minimizing the states of the DFAO from  Proposition~\ref{prop:eps-dfao} from $10$ to $7$,  see Figure \ref{fig:SE_mini}.
    \begin{figure}
        \centering
        \includegraphics[width=0.6\linewidth]{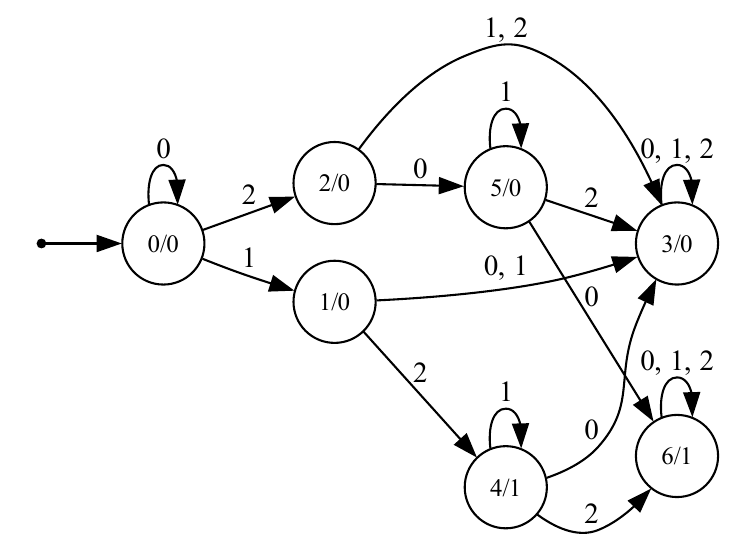}
        \caption{Minimized DFAO of that from Proposition~\ref{prop:eps-dfao} named \texttt{SE}.}
        \label{fig:SE_mini}
    \end{figure}
\end{remark}
Then, we verify that the identities from Proposition~\ref{prop:eps-auto} indeed hold for $\varepsilon$. For that, we assert that the code \texttt{pow3} accepts all powers of $3$ from indices $1$ onwards:\\
\texttt{reg pow3 msd\_3 "0*10*";}\\\\
Here, \texttt{msd\_3} asserts that all inputs are read in base $3$ with most significant digit first. Now, we check both the initial values, and the three identities prescribed in Proposition~\ref{prop:eps-auto} in order as follows. The first two lines check initial values $\varepsilon(n)=0$ for all $n<5$ and $\varepsilon(5)=1$ and the latter three check identities in Equation~\ref{eq:eps-rec}.\\
\texttt{eval init\_zero "?msd\_3 An ((n<5) => SE[n]=@0)";}\\
\texttt{eval init\_five "?msd\_3 SE[5]=@1";}\\
\texttt{eval rec0 "?msd\_3 Am ((m>=2) => (SE[3*m]=SE[m-1]))";}\\
\texttt{eval rec1 "?msd\_3 Am ((m>=2) => (SE[3*m+1]=SE[m]))";}\\
\texttt{eval rec2 "?msd\_3 Am ((m>=2) => (SE[3*m+2]=SE[m]))";}\\\\
Next, we check the truth of Lemma~\ref{lem:r3-gap} with the following code. Here, letting $p=3^{k-1}$ with $k\ge 2$, the prescribed gap becomes $3p-1\le n \le 5p$.\\
\texttt{def r3gap "?msd\_3 Ep (\$pow3(p) \& p>=3 \& 3*p-1<=n \& n<=5*p)";}\\
\texttt{eval ag3\_gap "?msd\_3 An (\$r3gap(n) => SE[n]=@0)";}\\
\begin{remark}
    The triadic denesting (Theorem~\ref{thm:r3}) and $3$-regularity (Corollary~\ref{cor:r3regular}) are established by classical means, since they concern the unbounded sequence $(a(n))$ and the finite generation of a $\Z$-module, neither of which is a first order statement about a finite-valued sequence. \texttt{Walnut} enters once the classical argument has reduced the problem to the binary correction sequence, whose defining recurrences it certifies in full. 
\end{remark}
\section{The case \texorpdfstring{$r=4$}{r=4}}\label{sec:r4}

The almost Golomb sequence of order $4$ is defined by $a(1)=1$, $a(k)=0$ for $k<1$,
and for $n\ge 2$, $a(n)$ is the smallest integer $\ge a(n-1)$ such that
\begin{equation}\label{eq:r4def}
  a\bigl(S_n\bigr) = n.
\end{equation}
where we set $S_n:=a(n)+a(n-1)+a(n-2)+a(n-3)$ as done previously.
This sequence is \oeis{A394219} in the OEIS, and its first terms are
\[
  1,2,2,2,3,3,4,4,5,6,6,7,7,8,8,9,9,9,10,10,11,11,11,12,\ldots.
\]

As in the triadic case, the denesting of Equation~\ref{eq:r4def} produces four formulas,
each with a binary correction sequence.
The cases $r=4$ and $r=5$ are handled by a hybrid method. The shape of the residue-class formulas is first identified computationally, by subtracting the relevant sliding window sums from the sequence and observing that the residuals are binary. The Boolean recurrences relating the resulting correctors are then derived by elimination, as illustrated in Appendix~\ref{app:r4-deriv}, and the complete system is certified for all $n$ by \texttt{Walnut}; see Appendix~\ref{app:dfao} for the identified DFAOs generating the required correction sequences.

For $n\ge 5$, define four correction sequences by
\begin{align}
  \varepsilon_0(n) &:= a(4n)   - \bigl(a(n-3)+a(n-2)+a(n-1)+a(n)\bigr) - 1, \label{eq:r4def0}\\
  \varepsilon_1(n) &:= a(4n+1) - \bigl(a(n-2)+a(n-1)+a(n)+a(n+1)\bigr), \label{eq:r4def1}\\
  \varepsilon_2(n) &:= a(4n+2) - \bigl(a(n-2)+a(n-1)+a(n)+a(n+1)\bigr), \label{eq:r4def2}\\
  \varepsilon_3(n) &:= a(4n+3) - \bigl(a(n-1)+a(n)+a(n+1)+a(n+2)\bigr) + 1. \label{eq:r4def3}
\end{align}

\begin{theorem}\label{thm:r4eps}
The correction sequences $\varepsilon_0,\varepsilon_1,\varepsilon_2,\varepsilon_3$ satisfy the following initial values for $5\le n\le 23$.
\[
\begin{array}{c|ccccccccccccccccccc}
n & 5&6&7&8&9&10&11&12&13&14&15&16&17&18&19&20&21&22&23\\ \hline
\varepsilon_0(n)&0&1&1&1&1&1&1&1&1&1&1&1&0&1&0&1&1&0&1\\
\varepsilon_1(n)&1&0&0&0&0&0&0&0&0&0&0&0&1&0&1&0&0&1&0\\
\varepsilon_2(n)&1&1&1&1&1&1&1&1&1&1&1&0&1&0&1&1&0&1&0\\
\varepsilon_3(n)&0&0&0&0&0&0&0&0&0&0&0&1&0&1&0&0&1&0&1
\end{array}
\]
Moreover, for every integer $m\ge 6$ the following $4$-adic recurrences hold:
\begin{align}
\varepsilon_0(4m)   &= (1-\varepsilon_0(m))\,\varepsilon_0(m+1)+\varepsilon_0(m)\,\varepsilon_0(m-1), \label{eq:e0_0}\\
\varepsilon_0(4m+1) &= (1-\varepsilon_0(m))\,\varepsilon_0(m-1)+\varepsilon_0(m)\,\varepsilon_0(m+1), \label{eq:e0_1}\\
\varepsilon_0(4m+2) &= \varepsilon_0(m), \label{eq:e0_2}\\
\varepsilon_0(4m+3) &= \varepsilon_0(m+1), \label{eq:e0_3}
\end{align}
\begin{align}
\varepsilon_1(4m)   &= \varepsilon_2(m)\,\bigl(\varepsilon_1(m-1)\oplus \varepsilon_1(m+1)\bigr), \label{eq:e1_0}\\
\varepsilon_1(4m+1) &= 1-\varepsilon_2(m), \label{eq:e1_1}\\
\varepsilon_1(4m+2) &= \varepsilon_1(m), \label{eq:e1_2}\\
\varepsilon_1(4m+3) &= \varepsilon_1(m+1), \label{eq:e1_3}
\end{align}
\begin{align}
\varepsilon_2(4m)   &= (1-\varepsilon_0(m-1))\,\varepsilon_1(m-1)+\varepsilon_0(m-1)\,(1-\varepsilon_1(m-1))\,\varepsilon_2(m), \label{eq:e2_0}\\
\varepsilon_2(4m+1) &= \varepsilon_2(m-1), \label{eq:e2_1}\\
\varepsilon_2(4m+2) &= \varepsilon_2(m), \label{eq:e2_2}\\
\varepsilon_2(4m+3) &= 1-\varepsilon_1(m), \label{eq:e2_3}
\end{align}
\begin{align}
\varepsilon_3(4m)   &= \varepsilon_2(m-1)\,\bigl(\varepsilon_2(m)\odot \varepsilon_1(m)\bigr), \label{eq:e3_0}\\
\varepsilon_3(4m+1) &= \varepsilon_3(m-1), \label{eq:e3_1}\\
\varepsilon_3(4m+2) &= \varepsilon_3(m), \label{eq:e3_2}\\
\varepsilon_3(4m+3) &= (1-\varepsilon_1(m+1))\,\bigl(\varepsilon_2(m)\odot \varepsilon_1(m)\bigr), \label{eq:e3_3}
\end{align}
where 
\[x\oplus y := x+y-2xy,\text{ and }\,\,x\odot y := 1-x-y+2xy\] respectively denote $\operatorname{XOR}$ and $\operatorname{XNOR}$
for bits $x,y\in\{0,1\}$.
In particular, $\varepsilon_i(n)\in\{0,1\}$ for all $n\ge 5$, and each $\varepsilon_i$
is $4$-automatic.
\end{theorem}

\begin{proof}
The claimed initial values follow from directly evaluating the definitions of $\varepsilon_0,\varepsilon_1,\varepsilon_2,\varepsilon_3$ given as Equations \ref{eq:r4def0}--\ref{eq:r4def3} respectively. Alternatively, the  following \texttt{Walnut} code also confirms them. Check Appendix \ref{app:dfao} for the sequences $\varepsilon_0,\varepsilon_1,\varepsilon_2,\varepsilon_3$ defined as DFAOs \texttt{C0}, \texttt{C1}, \texttt{C2}, \texttt{C3} respectively.\\
\texttt{eval init\_r4\_c0 "?msd\_4 An ((n<=23) => ((C0[n]=@1) <=> (n>5 \& n!=17 \& n!=19 \& n!=22)))";}\\
\texttt{eval init\_r4\_c1 "?msd\_4 An ((n<=23) => ((C1[n]=@1) <=> (n=5 | n=17 | n=19 |\\ n=22)))";}\\
\texttt{eval init\_r4\_c2 "?msd\_4 An ((n<=23) => ((C2[n]=@1) <=> (n>=5 \& n!=16 \& n!=18 \\\& n!=21 \& n!=23)))";}\\
\texttt{eval init\_r4\_c3 "?msd\_4 An ((n<=23) => ((C3[n]=@1) <=> (n=16 | n=18 | n=21 |\\ n=23)))";}\\

The $4$-adic recurrences~(\ref{eq:e0_0}--\ref{eq:e3_3}) were identified
by applying the defining equation \eqref{eq:r4def} at $n=4m,4m+1,4m+2,4m+3$
and expanding twice via Equations~\ref{eq:r4def0}--\ref{eq:r4def3}. Here, all $a(\cdot)$-terms
cancel, leaving Boolean polynomial identities among the bits
$\varepsilon_j(m+\delta)$ with $|\delta|\le 1$.
Two illustrative derivations are given in Appendix~\ref{app:r4-deriv}
to display the elimination mechanism. For a streamlined proof with \texttt{Walnut}, we first identify candidate DFAOs \texttt{C0}, \texttt{C1}, \texttt{C2}, \texttt{C3}  for sequences $\varepsilon_0$, $\varepsilon_1$, $\varepsilon_2$, $\varepsilon_3$ respectively, as listed in Appendix~\ref{app:dfao}. 
Then \texttt{Walnut} verifies, with the following code lines that translate the identities (\ref{eq:e0_0}--\ref{eq:e3_3}) respectively, that the sequences $\varepsilon_i$ indeed satisfy the claimed identities, and that their range is $\{0,1\}$. Consequently, the correction sequences $\varepsilon_i$ are $4$-automatic.\\
\texttt{eval r4\textunderscore e0\_0 "?msd\_4 A m (m>=6) => ((C0[4*m]=@1) <=> ((\textasciitilde(C0[m]=@1) \& \\(C0[m+1]=@1)) | ((C0[m]=@1) \& (C0[m-1]=@1))))";}\\
\texttt{eval r4\_e0\_1 "?msd\_4 A m (m>=6) => ((C0[4*m+1]=@1) <=> ((\textasciitilde(C0[m]=@1) \&\\ (C0[m-1]=@1)) | ((C0[m]=@1) \& (C0[m+1]=@1))))";}\\
\texttt{eval r4\_e0\_2 "?msd\_4 A m (m>=6) => ((C0[4*m+2]=@1) <=> (C0[m]=@1))";}\\
\texttt{eval r4\_e0\_3 "?msd\_4 A m (m>=6) => ((C0[4*m+3]=@1) <=> (C0[m+1]=@1))";}
\\\\
\texttt{eval r4\_e1\_0 "?msd\_4 A m (m>=6) => ((C1[4*m]=@1) <=> ((C2[m]=@1) \& \\(((C1[m-1]=@1) \& \textasciitilde(C1[m+1]=@1)) | (\textasciitilde(C1[m-1]=@1) \& (C1[m+1]=@1)))))";}\\
\texttt{eval r4\_e1\_1 "?msd\_4 A m (m>=6) => ((C1[4*m+1]=@1) <=> \textasciitilde(C2[m]=@1))";}\\
\texttt{eval r4\_e1\_2 "?msd\_4 A m (m>=6) => ((C1[4*m+2]=@1) <=> (C1[m]=@1))";}\\
\texttt{eval r4\_e1\_3 "?msd\_4 A m (m>=6) => ((C1[4*m+3]=@1) <=> (C1[m+1]=@1))";}\\
\\
\texttt{eval r4\_e2\_0 "?msd\_4 A m (m>=6) => ((C2[4*m]=@1) <=> ((\textasciitilde(C0[m-1]=@1) \&\\ (C1[m-1]=@1)) | ((C0[m-1]=@1) \& \textasciitilde(C1[m-1]=@1) \& (C2[m]=@1))))";}\\
\texttt{eval r4\_e2\_1 "?msd\_4 A m (m>=6) => ((C2[4*m+1]=@1) <=> (C2[m-1]=@1))";}\\
\texttt{eval r4\_e2\_2 "?msd\_4 A m (m>=6) => ((C2[4*m+2]=@1) <=> (C2[m]=@1))";}\\
\texttt{eval r4\_e2\_3 "?msd\_4 A m (m>=6) => ((C2[4*m+3]=@1) <=> \textasciitilde(C1[m]=@1))";}\\
\\
\texttt{eval r4\_e3\_0 "?msd\_4 A m (m>=6) => ((C3[4*m]=@1) <=> ((C2[m-1]=@1) \&\\ ((C2[m]=@1) <=> (C1[m]=@1))))";}\\
\texttt{eval r4\_e3\_1 "?msd\_4 A m (m>=6) => ((C3[4*m+1]=@1) <=> (C3[m-1]=@1))";}\\
\texttt{eval r4\_e3\_2 "?msd\_4 A m (m>=6) => ((C3[4*m+2]=@1) <=> (C3[m]=@1))";}\\
\texttt{eval r4\_e3\_3 "?msd\_4 A m (m>=6) => ((C3[4*m+3]=@1) <=> (\textasciitilde(C1[m+1]=@1) \&\\ ((C2[m]=@1) <=> (C1[m]=@1))))";}\\
Since every index occurring on the right-hand side of
(\ref{eq:e0_0}--\ref{eq:e3_3}) is strictly smaller than the index on the
left, the recurrences together with the initial values above determine
$\varepsilon_0,\ldots,\varepsilon_3$ uniquely for all $n\ge 5$. 
\end{proof}
\begin{remark}
    The corrector sequences $\varepsilon_0,\ldots,\varepsilon_3$ are defined by
    Equations~\ref{eq:r4def0}--\ref{eq:r4def3} only for $n\ge 5$. The values
    returned by the automata \texttt{C0} to \texttt{C3} for $n<5$ are therefore
    not values of the correctors; they are a by-product of the DFAO construction
    and play no role in the proof.
\end{remark}
\begin{corollary}\label{cor:r4}
The sequences $\varepsilon_0,\varepsilon_1,\varepsilon_2,\varepsilon_3$ of Theorem~\ref{thm:r4eps} are binary, and for all $n\ge 5$,
  \begin{align}
    a(4n)   &= a(n-3)+a(n-2)+a(n-1)+a(n)   + 1 + \varepsilon_0(n), \label{eq:r4a}\\
    a(4n+1) &= a(n-2)+a(n-1)+a(n)+a(n+1)   + \varepsilon_1(n),     \label{eq:r4b}\\
    a(4n+2) &= a(n-2)+a(n-1)+a(n)+a(n+1)   + \varepsilon_2(n),     \label{eq:r4c}\\
    a(4n+3) &= a(n-1)+a(n)+a(n+1)+a(n+2)   - 1 + \varepsilon_3(n). \label{eq:r4d}
  \end{align}
\end{corollary}
\begin{proof}
 The identities claimed are the definitions of correction sequences (\ref{eq:r4def0}--\ref{eq:r4def3}) rearranged and the correctors are binary by Theorem~\ref{thm:r4eps}.
\end{proof}
\begin{remark}
  The windows slide one step to the right as the residue increases from $0$ to $3$:
  the window for residue $0$ is $(n-3,n-2,n-1,n)$, for residues $1$ and $2$ it is
  $(n-2,n-1,n,n+1)$, and for residue $3$ it is $(n-1,n,n+1,n+2)$.
  This is the same half-step sliding pattern as in the triadic case.
\end{remark}

\begin{corollary}\label{cor:r4regular}
The almost Golomb sequence of order $4$ is $4$-regular.
\end{corollary}

\begin{proof}
Corollary~\ref{cor:r4} expresses $a(4n+i)$, the elements of the $4$-kernel of $(a(n))$, as a $\Z$-linear combination of
finitely many shifts of $a$ plus the $4$-automatic sequences $\varepsilon_i$
of Theorem~\ref{thm:r4eps}. Therefore, by~\cite[Theorem~2.5]{AS92}, $(a(n))$ is $4$-regular.
\end{proof}

\section{The case $r=5$}\label{sec:r5}

For $r=5$, the proof follows the same hybrid pattern as for $r=4$.
Indeed, one first identifies the correct pentadic correction sequences and
residue-class formulas, then uses \texttt{Walnut} to certify the resulting
automatic identities, yielding a complete proof of the denesting
formulas and $5$-regularity.

\subsection{Two binary correction sequences}

For $n\ge 3$, define two correction sequences by
\begin{equation}\label{eq:r5_def}
  \varepsilon(n) := a(5n) - S_{n} - 2,
  \qquad
  \eta(n) := S_{n+2} - 1 - a(5n+2),
\end{equation}
where $S_m:=a(m)+a(m-1)+a(m-2)+a(m-3)+a(m-4)$ with $a(k)=0$ for $k<1$.
Two further correction sequences enter the residues $3$ and $4$; they are
defined in the same way, for $n\ge 3$, by
\begin{equation}\label{eq:r5_def34}
  \theta(n) := a(5n+3) - S_{n+3} + 4,
  \qquad
  \varepsilon_4(n) := a(5n+4) - S_{n+4} + 6,
\end{equation}
and are treated in Theorem~\ref{thm:r5} below.

The two sequences $\varepsilon$ and $\eta$ are governed by the following $5$-adic recurrences.

\begin{theorem}\label{thm:r5eps}
The sequences $\varepsilon$ and $\eta$ take values in $\{0,1\}$,
satisfy $\varepsilon(n)\eta(n)=0$ for all $n\ge 3$, and are determined by
their values on $3\le n\le 19$, namely
\[
  \varepsilon(n)=1 \iff n\in\{4,13,15,17,19\},
  \qquad
  \eta(n)=1 \iff n\in\{3,12,14,16,18\},
\]
and $\varepsilon(n)=\eta(n)=0$ otherwise on that range,
together with the $5$-adic recurrences, valid for all $m\ge 4$:
\begin{align}
  \varepsilon(5m)   &= \varepsilon(m-1)\bigl(1-\varepsilon(m)\bigr), \label{eq:r5eps0}\\
  \varepsilon(5m+1) &= \varepsilon(m),\qquad \varepsilon(5m+3) = \varepsilon(m), \label{eq:r5eps13}\\
  \varepsilon(5m+2) &= \eta(m),\qquad\quad \varepsilon(5m+4) = \eta(m), \label{eq:r5eps24}\\
  \eta(5m)          &= \varepsilon(m),\qquad \eta(5m+2) = \varepsilon(m),\qquad
                       \eta(5m+4) = \varepsilon(m), \label{eq:r5eta024}\\
  \eta(5m+1)        &= \eta(m),\qquad\quad \eta(5m+3) = \eta(m). \label{eq:r5eta13}
\end{align}
In particular, the pair $U(n):=(\varepsilon(n),\eta(n))$ takes values in
$\{(0,0),(1,0),(0,1)\}$ and is $5$-automatic.
\end{theorem}

\begin{proof}
That $\varepsilon(n),\eta(n)\in\{0,1\}$ and $\varepsilon(n)\eta(n)=0$ hold for
$3\le n\le 19$ follows by direct computation from the definitions, and is
certified by \texttt{init\_r5} below.
The recurrences \eqref{eq:r5eps0}--\eqref{eq:r5eta13} were identified by applying
the defining equation $a(S_n)=n$ at $n=5m,\ldots,5m+4$, expanding each
$a(5m+j)$ in the window sum via Equation~\ref{eq:r5_def}, and cancelling the $a(\cdot)$
terms. Two illustrative derivations are given in Appendix~\ref{app:r5-deriv}
to display the elimination mechanism.
The candidate DFAO for $U(n)$ which we call \texttt{C5} is as given in Appendix~\ref{app:dfao}. Here, we have assigned $0$ for $(0,0)$, $1$ for $(1,0)$, and $2$ for $(0,1)$ for convenience of describing \texttt{C5} file. We can confirm the correction sequence values for $1\le n\le 4$ with the following \texttt{Walnut} code:\\
\texttt{eval init\_r5 "?msd\_5 An ((n>=3 \& n<=19) =>\\ ((C5[n]=@1 <=> (n=4|n=13|n=15|n=17|n=19))\\ \& (C5[n]=@2 <=> (n=3|n=12|n=14|n=16|n=18))))";}\\
Then, we
verify Identities~\ref{eq:r5eps0}--\ref{eq:r5eta13} by translating into the following code lines for $\varepsilon$ and $\eta$ respectively.\\
\texttt{eval r5\_eps\_0 "?msd\_5 A m (m>=4) => ((C5[5*m]=@1) <=> ((C5[m-1]=@1) \& \\\textasciitilde(C5[m]=@1)))";}\\
\texttt{eval r5\_eps\_1 "?msd\_5 A m (m>=4) => ((C5[5*m+1]=@1) <=> (C5[m]=@1))";}\\
\texttt{eval r5\_eps\_2 "?msd\_5 A m (m>=4) => ((C5[5*m+2]=@1) <=> (C5[m]=@2))";}\\
\texttt{eval r5\_eps\_3 "?msd\_5 A m (m>=4) => ((C5[5*m+3]=@1) <=> (C5[m]=@1))";}\\
\texttt{eval r5\_eps\_4 "?msd\_5 A m (m>=4) => ((C5[5*m+4]=@1) <=> (C5[m]=@2))";}\\
\\
\texttt{eval r5\_eta\_0 "?msd\_5 A m (m>=4) => ((C5[5*m]=@2) <=> (C5[m]=@1))";}\\
\texttt{eval r5\_eta\_1 "?msd\_5 A m (m>=4) => ((C5[5*m+1]=@2) <=> (C5[m]=@2))";}\\
\texttt{eval r5\_eta\_2 "?msd\_5 A m (m>=4) => ((C5[5*m+2]=@2) <=> (C5[m]=@1))";}\\
\texttt{eval r5\_eta\_3 "?msd\_5 A m (m>=4) => ((C5[5*m+3]=@2) <=> (C5[m]=@2))";}\\
\texttt{eval r5\_eta\_4 "?msd\_5 A m (m>=4) => ((C5[5*m+4]=@2) <=> (C5[m]=@1))";}\\\\
We now focus on the product
$\varepsilon(n)\eta(n)$ for all $n\ge 3$.

The right-hand side of each recurrence is either a constant, $\varepsilon(m)$,
$\eta(m)$, or $\varepsilon(m-1)(1-\varepsilon(m))$, hence lies in $\{0,1\}$.
Since the recurrences hold for $m\ge 4$ and every index on their right-hand
side is at least $3$, they determine $(\varepsilon,\eta)$ on all $n\ge 20$
from the values on $3\le n\le 19$ listed in the statement, which are
certified by \texttt{init\_r5}. The disjointness $\varepsilon\eta=0$ is propagated
by induction: Equations~\ref{eq:r5eps24} and~\ref{eq:r5eta024} exchange the roles of
$\varepsilon$ and $\eta$, while Equation~\ref{eq:r5eps0} produces $1$ only when
$\varepsilon(m-1)=1$, hence when $\eta(m-1)=0$, so $\eta(5m)=\varepsilon(m)=0$.
Finally, $U(5m+d)$ is determined by $U(m)$ and $U(m-1)$, so the Allouche--Shallit
criterion (Theorem~\ref{thm:AS}) applies and $U$ is $5$-automatic.
\end{proof}

\begin{remark}
The five transitions of the pair $(\varepsilon(m),\eta(m))$ at consecutive
indices suffice to cover all cases in $a_5$; see Appendix~\ref{app:r5-tables}.
\end{remark}

\medskip
\begin{center}
\renewcommand{\arraystretch}{1.3}
\begin{tabular}{c|cccccccccc}
$n$ & 1 & 2 & 3 & 4 & 5 & 6 & 7 & 8 & 9 & 10 \\\hline
$\varepsilon(n)$ & 0 & 0 & 0 & 1 & 0 & 0 & 0 & 0 & 0 & 0 \\
$\eta(n)$ & 0 & 0 & 1 & 0 & 0 & 0 & 0 & 0 & 0 & 0 \\
\hline
$n$ & 11 & 12 & 13 & 14 & 15 & 16 & 17 & 18 & 19 & 20 \\\hline
$\varepsilon(n)$ & 0 & 0 & 1 & 0 & 1 & 0 & 1 & 0 & 1 & 0 \\
$\eta(n)$ & 0 & 1 & 0 & 1 & 0 & 1 & 0 & 1 & 0 & 1 \\
\end{tabular}
\end{center}
\medskip
\subsection{The denesting theorem}

Set $\sigma(n):=\varepsilon(n)+\eta(n)\in\{0,1\}$.
For residues $3$ and $4$ two further correctors appear, both expressible from
$(\varepsilon,\eta)$; their formulas are listed in Appendix~\ref{app:r5-tables}.

With these sequences in hand, the five denesting formulas are immediate.

\begin{theorem}\label{thm:r5}
For all $n\ge 3$, the almost Golomb sequence of order $5$ satisfies:
\begin{align}
  a(5n)   &= S_n   + 2 + \varepsilon(n),   \label{eq:r5a0}\\
  a(5n+1) &= S_{n+1} + 1 - \varepsilon(n),   \label{eq:r5a1}\\
  a(5n+2) &= S_{n+2} - 1 - \eta(n),          \label{eq:r5a2}\\
  a(5n+3) &= S_{n+3} - 4 + \theta(n),        \label{eq:r5a3}\\
  a(5n+4) &= S_{n+4} - 6 + \varepsilon_4(n), \label{eq:r5a4}
\end{align}
where $\theta(n)\in\{0,1\}$ and $\varepsilon_4(n)\in\{0,1,2\}$ are both
$5$-automatic, expressed as explicit functions of $U$ in
Appendix~\ref{app:r5-tables}. Consequently, $(a(n))$ is $5$-regular
by~\cite[Theorem~2.5]{AS92}.
\end{theorem}

\begin{proof}
Formulas \eqref{eq:r5a0} and \eqref{eq:r5a2} are the definitions
\eqref{eq:r5_def}. Formula \eqref{eq:r5a1} is derived in
Appendix~\ref{app:r5-deriv} by the same elimination used in Theorem~\ref{thm:r5eps}.
Formulas \eqref{eq:r5a3} and \eqref{eq:r5a4} are the definitions
\eqref{eq:r5_def34} of $\theta$ and $\varepsilon_4$.

The $5$-automaticity of $\theta$ and $\varepsilon_4$ follows from the explicit
tables in Appendix~\ref{app:r5-tables}: both are obtained from the $5$-automatic
sequence $U$ by a finite-state post-processing (the table defines a finite
transducer whose input is the pair $(U(m-1),U(m))$ or $(U(m),U(m+1))$).

The $5$-regularity of $(a(n))$ follows because Equations~\ref{eq:r5a0}--\ref{eq:r5a4}
express each subsequence $(a(5n+i))_{n\ge 0}$ as a $\Z$-linear combination of
finitely many shifts of $(a(n))$ plus $5$-automatic sequences; by
\cite[Theorem~2.5]{AS92} this implies $(a(n))$ is $5$-regular.
\end{proof}

\section{Non-convergence of $a(n)/n$}\label{sec:nonconv}

A common feature of all almost Golomb sequences is that $a(n)/n$ oscillates
between two explicit rational limits, failing to converge. This contrasts
sharply with Golomb's sequence, where $a(n)/n^{\varphi-1}\to\varphi^{2-\varphi}$
converges to a smooth constant.

\subsection{The case $r=2$}

\begin{theorem}\label{thm:ratio2}
  Let $(a(n))$ be the almost Golomb sequence of order $2$.
  \begin{enumerate}[label=(\alph*)]
    \item $a(n)/n$ does not converge: $a(2^k)/2^k = 3/4$ and
          $a(3\cdot 2^{k-1})/(3\cdot 2^{k-1}) = 2/3$ for all $k\ge 2$.
    \item The Cesàro means $C_N = \frac{1}{N}\sum_{n=1}^N \frac{a(n)}{n}$
          do not converge. They have two limit points:
          \[
            L_1 = \lim_{k\to\infty} C_{2^k}
                = \tfrac{3}{4} + \log\!\left(\tfrac{6^{3/4}}{4}\right),
            \qquad
            L_2 = \lim_{k\to\infty} C_{3\cdot 2^{k-1}}
                = \tfrac{2}{3} + \log\!\left(\tfrac{3^{2/3}}{2}\right).
          \]
  \end{enumerate}
\end{theorem}

\begin{proof}
  Part (a) follows immediately from Corollary~\ref{cor:pivots}.

  For part (b), introduce the dyadic block average
  $A_k = 2^{-k}\sum_{n=2^k}^{2^{k+1}-1} a(n)/n$.
  Writing $n=2^k+j$ and $t=j/2^k\in[0,1)$, Proposition~\ref{prop:dyadic} gives
  \[
    \frac{a(n)}{n} =
    \begin{cases}
      \dfrac{\tfrac{3}{4}+\tfrac{t}{2}}{1+t} + O(2^{-k}) & 0\le t\le \tfrac{1}{2},\\[8pt]
      \dfrac{\tfrac{1}{2}+t}{1+t}                         & \tfrac{1}{2}<t<1.
    \end{cases}
  \]
  Summing as a Riemann integral gives $A_k\to\alpha$ where
  \[
    \alpha = \int_0^{1/2}\frac{\tfrac{3}{4}+\tfrac{t}{2}}{1+t}\,dt
           + \int_{1/2}^1  \frac{\tfrac{1}{2}+t}{1+t}\,dt
           = \tfrac{3}{4}+\log\!\left(\tfrac{6^{3/4}}{4}\right).
  \]
  Decomposing $C_{2^k}$ into dyadic blocks and using the fact that
  the weighted average $\sum_{j=0}^{k-1} 2^{j-k} A_j$ converges to $\alpha$
  whenever each $A_j\to\alpha$ (a standard Cesàro-type argument,
  since the weights $2^{j-k}$ are nonnegative and sum to $1-2^{-k}\to 1$)
  gives $\lim C_{2^k}=\alpha=L_1$.

  For $N_k=3\cdot 2^{k-1}$, write
  $C_{N_k} = \frac{2}{3}C_{2^k} + \frac{1}{3}D_k$
  where $D_k = 2^{1-k}\sum_{n=2^k+1}^{3\cdot 2^{k-1}} a(n)/n$.
  Since this sum lies in the first half of the dyadic block ($t\in[0,1/2]$),
  $D_k\to\beta$ with
  \[
    \beta = 2\int_0^{1/2}\frac{\tfrac{3}{4}+\tfrac{t}{2}}{1+t}\,dt
          = \tfrac{1}{2}+\tfrac{1}{2}\log\!\left(\tfrac{3}{2}\right).
  \]
  Hence $\lim C_{N_k} = \frac{2}{3}\alpha+\frac{1}{3}\beta = L_2$.
  Finally $L_1-L_2 = (1-\log(8/3))/12 > 0$ since $e>8/3$.
\end{proof}

The run structure of the sequence is also captured automatically.

\begin{proposition}
  The sequence $\Delta a(n) = a(n+1)-a(n) \in \{0,1\}$ is $2$-automatic.
\end{proposition}

\begin{proof}
  Proposition~\ref{prop:dyadic} shows that in the block $[2^k, 2^{k+1})$
  written as $n=2^k+j$:
  $\Delta a(n)=0$ when $j$ is odd and $j\le 2^{k-1}-1$
  (the ceiling-based formula produces repeated values at consecutive even/odd pairs);
  $\Delta a(n)=1$ for all $n\ge 3\cdot 2^{k-1}$ (the second half of the block
  has slope $1$).
  This is a regular language on the binary representation of $n$, so
  $\Delta a$ is $2$-automatic.
\end{proof}

The multiplicity sequence of values in $(a(n))$ is also $2$-automatic.

\begin{proposition}\label{prop:N2_multiplicity}
  Let $N_2(n)$ denote the number of times $n$ appears in the almost Golomb
  sequence of order $2$. Then $N_2(n)\in\{1,2\}$ for all $n\ge 1$, and the
  sequence $(N_2(n))_{n\ge 1}$ is $2$-automatic, satisfying the recurrences
  \begin{equation}\label{eq:N2rec}
    N_2(2n) = N_2(n),\qquad N_2(2n+1) = N_2(n+1)\qquad (n\ge 2),
  \end{equation}
  with initial values $N_2(1)=1$, $N_2(2)=2$, $N_2(3)=1$.
  Equivalently, $N_2$ is given by the single recursion
  \begin{equation}\label{eq:N2formula}
    N_2(n) = \begin{cases}
      2 - (n \bmod 2) & \text{if } n \le 3, \\
      N_2\!\bigl(\lceil n/2 \rceil\bigr) & \text{if } n \ge 4.
    \end{cases}
  \end{equation}
\end{proposition}

\begin{proof}
  From Proposition~\ref{prop:dyadic}, in the dyadic block $[2^k, 2^{k+1})$
  written as $n=2^k+j$ with $0\le j<2^k$:
  \begin{itemize}
    \item First half ($0\le j\le 2^{k-1}$): $a(n)=3\cdot 2^{k-2}+\lceil j/2\rceil$.
    The value $3\cdot 2^{k-2}$ appears once ($j=0$); each value in
    $[3\cdot 2^{k-2}+1,\,2^k]$ appears exactly twice (two consecutive values of $j$).
    \item Second half ($2^{k-1}<j<2^k$): $a(n)=n-2^{k-1}$, so
    each value in $[2^k+1,\,3\cdot 2^{k-1}-1]$ appears exactly once.
  \end{itemize}
  Hence $N_2(v)=2$ if and only if $v\in(3\cdot 2^{k-2},\,2^k]$ for some $k\ge 1$.

  We check that Equation~\ref{eq:N2formula} agrees with this. The map
  $v\mapsto\lceil v/2\rceil$ descends until reaching $\{1,2,3\}$.
  Indeed, $v\in(3\cdot 2^{k-2},\,2^k]$ iff the descent
  $v\to\lceil v/2\rceil\to\lceil\lceil v/2\rceil/2\rceil\to\cdots$
  terminates at $2$: if $v$ is in the upper quarter of its dyadic block,
  $\lceil v/2\rceil$ lands in the upper quarter of the next smaller block,
  and the property is preserved until reaching $2\in(3/2,2]$.
  Conversely if $v$ is in the lower three-quarters, the descent
  eventually reaches $1$ or $3$, both of which give $N_2=1$ via Equation~\ref{eq:N2formula}.

  The recurrences~\eqref{eq:N2rec} are a direct consequence of
  Equation~\ref{eq:N2formula}: for $n\ge 4$,
  $N_2(2n)=N_2(\lceil 2n/2\rceil)=N_2(n)$, and
  $N_2(2n+1)=N_2(\lceil(2n+1)/2\rceil)=N_2(n+1)$.
  Since $N_2$ takes values in $\{1,2\}$, it is bounded.
  The recurrences show that every subsequence in the $2$-kernel of $N_2$
  is a $\mathbb{Z}$-linear combination of
  $(N_2(n))_{n\ge 1}$, $(N_2(n+1))_{n\ge 1}$, and the constant sequence~$1$.
  Hence $N_2$ is $2$-regular. Since $N_2$ takes values in the finite
  set $\{1,2\}$, $2$-regularity implies $2$-automaticity~\cite{AS92}.
\end{proof}

\subsection{The case $r=3$}

To identify the two pivotal subsequences, we first locate them outside the correction set.

\begin{lemma}\label{lem:gaps}
  For $m\in\{5,8\}$ and all $k\ge 1$:\; $\varepsilon(m\cdot 3^k)=\varepsilon(m\cdot 3^k-1)=0$.
\end{lemma}

\begin{proof}
  Recall that $I_0=\{5\}$ and $I_k=[\ell_k,u_k]\cap\mathbb{Z}$ with
  $\ell_k=\lceil(11\cdot 3^k-1)/2\rceil$ and $u_k=\lfloor(13\cdot 3^k-3)/2\rfloor$
  for $k\ge 1$.
  The intervals are ordered by scale: $u_{k-1}<\ell_k$ for all $k\ge 1$, since
  $u_{k-1}=(13\cdot 3^{k-1}-3)/2$ and $\ell_k=(11\cdot 3^k-1)/2$ give
  $\ell_k - u_{k-1} = (11\cdot 3^k - 13\cdot 3^{k-1} + 2)/2 = (20\cdot 3^{k-1}+2)/2 > 0$.
  Hence the $I_k$ are pairwise disjoint, and $n\in\mathcal{I}$
  iff $n\in I_k$ for a unique $k$.

  For $m=5$ and $k\ge 1$, we show $5\cdot 3^k$ lies strictly between $I_{k-1}$ and $I_k$.
  On the left: $5\cdot 3^k > u_{k-1} = (13\cdot 3^{k-1}-3)/2$
  since $10\cdot 3^k = 30\cdot 3^{k-1} > 13\cdot 3^{k-1} - 3$.
  On the right: $5\cdot 3^k < \ell_k = (11\cdot 3^k-1)/2$
  since $10\cdot 3^k < 11\cdot 3^k - 1$, i.e.\ $1 < 3^k$, which holds for $k\ge 1$.
  Hence $5\cdot 3^k\notin\mathcal{I}$, and since $5\cdot 3^k - 1 < 5\cdot 3^k < \ell_k$
  and $5\cdot 3^k - 1 > u_{k-1}$, we get $5\cdot 3^k - 1\notin\mathcal{I}$ as well.

  For $m=8$ and $k\ge 1$, we show $8\cdot 3^k$ lies strictly between $I_k$ and $I_{k+1}$.
  On the left: $8\cdot 3^k > u_k = (13\cdot 3^k-3)/2$
  since $16\cdot 3^k > 13\cdot 3^k - 3$.
  On the right: $8\cdot 3^k < \ell_{k+1} = (11\cdot 3^{k+1}-1)/2$
  since $16\cdot 3^k < 33\cdot 3^k - 1$.
  Hence $8\cdot 3^k\notin\mathcal{I}$, and since $u_k < 8\cdot 3^k - 1 < 8\cdot 3^k < \ell_{k+1}$,
  we get $8\cdot 3^k - 1\notin\mathcal{I}$.
  Therefore $\varepsilon(5\cdot 3^k)=\varepsilon(5\cdot 3^k-1)
  =\varepsilon(8\cdot 3^k)=\varepsilon(8\cdot 3^k-1)=0$.
\end{proof}
\begin{remark}
    The following \texttt{Walnut} code reproves Lemma~\ref{lem:gaps} quite efficiently. Here, we take $p=3^k$ with $k\ge 1$.\\
\texttt{def fam50 "?msd\_3 Ep (\$pow3(p) \& p>=3 \& n=5*p)";}\\
\texttt{def fam51 "?msd\_3 Ep (\$pow3(p) \& p>=3 \& n+1=5*p)";}\\
\texttt{def fam80 "?msd\_3 Ep (\$pow3(p) \& p>=3 \& n=8*p)";}\\
\texttt{def fam81 "?msd\_3 Ep (\$pow3(p) \& p>=3 \& n+1=8*p)";}\\\\
\texttt{eval ag3\_fam50 "?msd\_3 An (\$fam50(n) => SE[n]=@0)";}\\
\texttt{eval ag3\_fam51 "?msd\_3 An (\$fam51(n) => SE[n]=@0)";}\\
\texttt{eval ag3\_fam80 "?msd\_3 An (\$fam80(n) => SE[n]=@0)";}\\
\texttt{eval ag3\_fam81 "?msd\_3 An (\$fam81(n) => SE[n]=@0)";}
\end{remark}
\begin{theorem}\label{thm:ratio3}
  Let $(a(n))$ be the almost Golomb sequence of order $3$.
  For all $k\ge 1$:
  \[
    a(5\cdot 3^k) = 8\cdot 3^{k-1}, \quad
    \frac{a(5\cdot 3^k)}{5\cdot 3^k} = \frac{8}{15};
    \qquad
    a(8\cdot 3^k) = 15\cdot 3^{k-1}, \quad
    \frac{a(8\cdot 3^k)}{8\cdot 3^k} = \frac{5}{8}.
  \]
  Thus $a(n)/n$ does not converge.
\end{theorem}

\begin{proof}
  We propagate the local pattern around $5\cdot 3^k$ by induction on $k$.
  The base case $k=1$ is verified directly (pattern around $15$).
  Assuming for some $k\ge 1$:
  \[
    a(5\cdot 3^k-2)=8\cdot 3^{k-1}-1,\;\;
    a(5\cdot 3^k-1)=8\cdot 3^{k-1},\;\;
    a(5\cdot 3^k)  =8\cdot 3^{k-1},\;\;
    a(5\cdot 3^k+1)=8\cdot 3^{k-1}+1,
  \]
  set $n=5\cdot 3^k$. By Lemma~\ref{lem:gaps}, $\varepsilon(n)=\varepsilon(n-1)=0$.
  Applying Equations~\ref{eq:r3a}--\ref{eq:r3c} (and the analogues with $n-1$):
  \begin{align*}
    a(3n-2) &= a(n-2)+a(n-1)+a(n)   = 8\cdot 3^k-1,\\
    a(3n-1) &= a(n-1)+a(n)+a(n+1)-1 = 8\cdot 3^k,\\
    a(3n)   &= a(n-2)+a(n-1)+a(n)+1 = 8\cdot 3^k,\\
    a(3n+1) &= a(n-1)+a(n)+a(n+1)   = 8\cdot 3^k+1,
  \end{align*}
  which is the same pattern at rank $k+1$ around $3n=5\cdot 3^{k+1}$.

  For the family $a(8\cdot 3^k)=15\cdot 3^{k-1}$, we propagate the local pattern
  \[
    a(8\cdot 3^k-2)=15\cdot 3^{k-1}-1,\quad
    a(8\cdot 3^k-1)=a(8\cdot 3^k)=15\cdot 3^{k-1},\quad
    a(8\cdot 3^k+1)=15\cdot 3^{k-1}+1.
  \]
  The base case $k=1$ is read off directly: $a(22)=14$, $a(23)=a(24)=15$, $a(25)=16$.
  Assuming the pattern at rank $k$, set $n=8\cdot 3^k$.
  By Lemma~\ref{lem:gaps}, $\varepsilon(n)=\varepsilon(n-1)=0$.
  Applying Equations~\ref{eq:r3a}--\ref{eq:r3c} at $n-1$ and $n$:
  \begin{align*}
    a(3n-2) &= a(n-2)+a(n-1)+a(n)   = 15\cdot 3^k-1,\\
    a(3n-1) &= a(n-1)+a(n)+a(n+1)-1 = 15\cdot 3^k,\\
    a(3n)   &= a(n-2)+a(n-1)+a(n)+1 = 15\cdot 3^k,\\
    a(3n+1) &= a(n-1)+a(n)+a(n+1)   = 15\cdot 3^k+1,
  \end{align*}
  which is the same pattern at rank $k+1$ around $3n=8\cdot 3^{k+1}$.
  Since $8/15\ne 5/8$, the ratio $a(n)/n$ does not converge.
\end{proof}

\subsection{The case $r=4$}

The quaternary case exhibits the same oscillatory behavior as the triadic case,
but the two pivotal subsequences are now identified via the affine recurrence
$u_{k+1}=4u_k - 2$. The following observation will be used to evaluate automatic sequences along geometric progressions.

\begin{lemma}\label{lem:auto-geometric}
Let $u$ be a $b$-automatic sequence generated by a DFAO reading
base-$b$ expansions from most significant to least significant digit,
and let $P$, $Q$ be fixed base-$b$ words.
Then the sequence $k\mapsto u([P\,0^k\,Q]_b)$ is ultimately periodic.
\end{lemma}

\begin{proof}
Let $(\mathcal Q,\Sigma,\delta,q_0,\tau)$ be the DFAO.
After reading the prefix $P$, the automaton is in a fixed state $q_P$.
Reading $0^k$ applies the map $q\mapsto\delta(q,0)$ exactly $k$ times,
producing a sequence of states that is ultimately periodic
(since the state set is finite).
Reading the suffix $Q$ and applying $\tau$ preserves this periodicity.
\end{proof}
The non-convergence results below (Theorems~\ref{thm:ratio4} and~\ref{thm:ratio5}) rely on Lemma~\ref{lem:auto-geometric} applied to the recurrence systems of Theorems~\ref{thm:r4eps} and~\ref{thm:r5eps}; the DFAO data in Appendix~\ref{app:dfao} is required for their \texttt{Walnut}-based proof.

\begin{theorem}\label{thm:ratio4}
  Let $(a(n))$ be the almost Golomb sequence of order $4$.
  The sequence $a(n)/n$ does not converge. Setting $A_k=a(4^k)$,
  the recurrence
  \begin{equation}\label{eq:r4-rec}
    A_{k+1} = 4A_k - 2
  \end{equation}
  holds for all $k\ge 3$, with general solution $A_k = C\cdot 4^k + \tfrac{2}{3}$.
  The two families
  \[
    \frac{a(4^k)}{4^k} \to \frac{25}{48}, \qquad
    \frac{a(7\cdot 4^k)}{7\cdot 4^k} \to \frac{10}{21}
  \]
  have distinct limits, so $a(n)/n$ has at least two limit points.
\end{theorem}

\begin{proof}
  Since $4^k=[1\,0^k]_4$ and $7\cdot 4^k=[13\,0^k]_4$ in base~$4$,
  Lemma~\ref{lem:auto-geometric} implies that the correction patterns along
  these two families are ultimately periodic.
  Machine evaluation of the certified recurrence system
  of Theorem~\ref{thm:r4eps} yields, from $k=3$ onward,
  \[
    (\varepsilon_0(4^k),\,\varepsilon_1(4^k),\,\varepsilon_2(4^k),\,\varepsilon_3(4^k))
    =(1,0,0,1),
  \]
  together with the fixed local neighborhood pattern
  \[
    a(4^k+j) = A_k + \delta(j) \qquad (-4\le j\le 4),
  \]
  where
  \[
    \delta(j) = \begin{cases}
      -2 & j\in\{-4,-3\},\\
      -1 & j\in\{-2,-1\},\\
      \phantom{-}0 & j\in\{0,1,2\},\\
      +1 & j\in\{3,4\}.
    \end{cases}
  \]
  Substituting into Equation~\ref{eq:r4a} with $n=4^k$ gives
  $A_{k+1} = 4A_k - 2$.
  With $A_3 = a(64) = 34$, the solution is
  $A_k = \tfrac{25}{48}\cdot 4^k + \tfrac{2}{3}$.
  The same evaluation along $7\cdot 4^k$, with $B_k:=a(7\cdot 4^k)$, gives
  $B_k/(7\cdot 4^k)\to 10/21$.
  Since $25/48 \ne 10/21$, the ratio $a(n)/n$ does not converge.
  (The local neighborhood patterns can be verified by
  iterating the certified recurrences of Theorem~\ref{thm:r4eps}.)
\end{proof}

\begin{remark}
  The affine recurrence $u_{k+1}=4u_k-2$ (with fixed point $2/3$) contrasts with
  the purely multiplicative recurrences of orders $2$ and $3$.
  Consequently the two limiting ratios $25/48$ and $10/21$ do not have the
  ``$p/q$ and $q/p$'' symmetry seen in the triadic case ($8/15$ and $5/8$).
\end{remark}

\subsection{The case $r=5$}

\begin{theorem}\label{thm:ratio5}
The ratio $a(n)/n$ for the almost Golomb sequence of order $5$ does not converge.
Setting $A_k=a(5^k)$ and $B_k=a(2\cdot 5^k)$, one has for all $k\ge 2$:
\begin{equation}\label{eq:r5rec}
  A_{k+1} =
  \begin{cases}
    5A_k - 1, & k\ \text{even},\\
    5A_k - 4, & k\ \text{odd},
  \end{cases}
  \qquad
  B_{k+1} = 5B_k - 1.
\end{equation}
Consequently
\[
  \lim_{\substack{k\to\infty\\k\text{ even}}} \frac{a(5^k)}{5^k} = \frac{93}{200},
  \qquad
  \lim_{k\to\infty}\frac{a(2\cdot 5^k)}{2\cdot 5^k} = \frac{87}{200},
\]
and these two limits are distinct.
\end{theorem}

\begin{proof}
Since $5^k=[1\,0^k]_5$ and $2\cdot 5^k=[2\,0^k]_5$ in base~$5$,
Lemma~\ref{lem:auto-geometric} implies that the correction patterns along
these families are ultimately periodic.
Machine evaluation of the certified recurrence system of
Theorem~\ref{thm:r5eps} yields, for $k\ge 2$, the alternating pattern
$U(5^k)\in\{(1,0),(0,1)\}$ with period~$2$, and a constant correction pattern
along $2\cdot 5^k$.
The same evaluation gives the corresponding local neighborhood patterns of $a$,
from which the affine recurrences~\eqref{eq:r5rec} follow by substitution into
Equations~\ref{eq:r5a0}--\ref{eq:r5a4}.
For even $k$, $A_k/5^k \to 93/200$.
For $B_k$, $B_k/(2\cdot 5^k)\to 87/200$.
Since $93/200\ne 87/200$, the ratio $a(n)/n$
has at least two limit points.
\end{proof}


\subsection{Summary}

\medskip
\begin{center}
\renewcommand{\arraystretch}{1.6}
\begin{tabular}{c|l|c|l|c}
$r$ & Pivotal family 1 & Limit & Pivotal family 2 & Limit \\\hline
$2$ & $2^k$ & $\dfrac{3}{4}$ & $3\cdot 2^{k-1}$ & $\dfrac{2}{3}$ \\[4pt]
$3$ & $5\cdot 3^k$ & $\dfrac{8}{15}$ & $8\cdot 3^k$ & $\dfrac{5}{8}$ \\[4pt]
$4$ & $4^k$ & $\dfrac{25}{48}$ & $7\cdot 4^k$ & $\dfrac{10}{21}$ \\[4pt]
$5$ & $5^k$\;($k$ even) & $\dfrac{93}{200}$ & $2\cdot 5^k$ & $\dfrac{87}{200}$ \\
\end{tabular}
\end{center}
\medskip

In each case both limits lie in $(0,1)$, confirming linear growth $a(n)\asymp n$
(meaning $C_1 n \le a(n) \le C_2 n$ for positive constants $C_1,C_2$)
with an oscillating prefactor. For $r=2$, the Ces\`aro means also fail to converge
(Theorem~\ref{thm:ratio2}); the behavior for $r\ge 3$ is open
(see Section~\ref{sec:open}).

\section{Combinatorial interpretations}\label{sec:combinatorial}

The denesting formulas of Sections~\ref{sec:r2} and~\ref{sec:r3}
encode algebraic identities. In this section we extract from them
three combinatorial properties that reveal the
structure of the difference sequence $d(n)=a(n+1)-a(n)\in\{0,1\}$
in the cases $r=2$ and $r=3$.

\subsection{A second-bit characterization of multiplicities ($r=2$)}

Proposition~\ref{prop:N2_multiplicity} shows that the multiplicity
$N_2(m)$ depends only on the iterated ceiling map
$m\mapsto\lceil m/2\rceil$. We now translate this into a direct
binary characterization.

\begin{proposition}\label{prop:second-bit}
For every $m\ge 4$, write $m-1=(1\,b_2\,b_3\cdots b_\ell)_2$ with
$b_2$ denoting the second most significant bit. Then
\[
  N_2(m)=\begin{cases}
    2 & \text{if } b_2=1,\\
    1 & \text{if } b_2=0.
  \end{cases}
\]
\end{proposition}

\begin{proof}
By Proposition~\ref{prop:N2_multiplicity},
$N_2(m)=N_2(\lceil m/2\rceil)$ for $m\ge 4$.
Set $k=m-1\ge 3$ and write $k=(1\,b_2\cdots b_\ell)_2$.
The operation $m\mapsto\lceil m/2\rceil$ translates to
$k\mapsto\lfloor k/2\rfloor$, which erases the least significant bit.
After $\ell-2$ iterations the two-digit prefix
$(1\,b_2)_2$ is reached, corresponding to
$m'=k'+1$ where $k'\in\{2,3\}$.
If $b_2=1$: the descent reaches $k'=3$, hence $m'=4$, $N_2(4)=N_2(2)=2$.
If $b_2=0$: the descent reaches $k'=2$, hence $m'=3$, $N_2(3)=1$.
\end{proof}

\begin{remark}
The almost Golomb sequence of order $2$ is thus the unique nondecreasing
sequence of positive integers in which each value $m\ge 4$ is repeated
twice if the second most significant bit of $(m-1)$ is $1$, and once otherwise.
The asymptotic density of values appearing twice is
$\lim_{N\to\infty}|\{m\le N:N_2(m)=2\}|/N=1/2$,
which governs the oscillation of $a(n)/n$ between $2/3$ and $3/4$.
\end{remark}

\subsection{A Boolean cellular automaton ($r=2$)}

The denesting theorem (Theorem~\ref{thm:r2}) translates into a
remarkably simple rule for the difference sequence.

\begin{proposition}\label{prop:AND-automaton}
The difference sequence $d(n)=a(n+1)-a(n)\in\{0,1\}$ of the almost
Golomb sequence of order $2$ satisfies, for all $n\ge 2$:
\begin{equation}\label{eq:AND}
  d(2n)=1,\qquad d(2n+1)=d(n)\cdot d(n+1).
\end{equation}
\end{proposition}

\begin{proof}
From Theorem~\ref{thm:r2} (valid for $n\ge 2$):
\begin{align*}
  d(2n) &= a(2n+1)-a(2n) = \bigl[a(n)+a(n+1)\bigr]-\bigl[a(n)+a(n+1)-1\bigr]=1,\\
  d(2n+1) &= a(2n+2)-a(2n+1)
           = \bigl[a(n+1)+a(n+2)-1\bigr]-\bigl[a(n)+a(n+1)\bigr]\\
           &= d(n+1)+d(n)-1.
\end{align*}
Since $d\in\{0,1\}$ and $N_2(m)\in\{1,2\}$
(Proposition~\ref{prop:N2_multiplicity}), run lengths are at most~$2$,
so $d(n)+d(n+1)\ge 1$ for all $n\ge 2$.
Hence $d(n)+d(n+1)-1\in\{0,1\}$, and the identity
$x+y-1=xy$ holds for all $x,y\in\{0,1\}$ with $x+y\ge 1$.
\end{proof}

\begin{remark}
Formula~\eqref{eq:AND} shows that $d$ is generated by a multiplicative
cellular automaton: every even position carries the value $1$,
and every odd position is the Boolean AND of its two ``parent'' values
at positions $n$ and $n+1$.
The sequence $a(n)=\sum_{k=0}^{n-1}d(k)$ is thus the
accumulation function of this automaton.
\end{remark}

\subsection{A palindromic substitution rule ($r=3$)}

For $r=3$, the triadic structure of the difference sequence
admits a palindromic decomposition.

\begin{proposition}\label{prop:palindromic}
For every $n\ge 3$, define the triadic block
$B_n=(d(3n-2),\,d(3n-1),\,d(3n))$.
Then:
\begin{enumerate}[label=(\alph*)]
  \item $B_n$ is palindromic: $d(3n-2)=d(3n)$.
  \item The sum $d(3n-2)+d(3n-1)+d(3n)=L_n$.
  \item The block $B_n$ is determined by $L_n$ as follows:
    \[
      B_n=\begin{cases}
        (0,1,0) & \text{if } L_n=1,\\
        (1,0,1) & \text{if } L_n=2,\\
        (1,1,1) & \text{if } L_n=3.
      \end{cases}
    \]
    In particular, $L_n=3$ if and only if $\varepsilon(n-1)=1$.
\end{enumerate}
\end{proposition}

\begin{proof}
Using the denesting formulas \eqref{eq:r3a}--\eqref{eq:r3c}
at ranks $n-1$ and $n$:
\begin{align*}
  d(3n-2) &= a(3(n{-}1){+}2)-a(3(n{-}1){+}1) \\
          &= \bigl[a(n{-}1)+a(n)+a(n{+}1)-1-\varepsilon(n{-}1)\bigr]
             - \bigl[a(n{-}2)+a(n{-}1)+a(n)\bigr]\\
          &= a(n{+}1)-a(n{-}2)-1-\varepsilon(n{-}1) = L_n-1-\varepsilon(n{-}1).
\end{align*}
By the same method:
\begin{align*}
  d(3n-1) &= a(3n)-a(3(n{-}1){+}2) \\
          &= \bigl[a(n{-}2)+a(n{-}1)+a(n)+1+\varepsilon(n{-}1)\bigr]\\
          &\quad - \bigl[a(n{-}1)+a(n)+a(n{+}1)-1-\varepsilon(n{-}1)\bigr]\\
          &= 2-L_n+2\varepsilon(n{-}1).
\end{align*}
And:
\[
  d(3n) = a(3n{+}1)-a(3n) = L_n-1-\varepsilon(n{-}1) = d(3n{-}2).
\]
This proves (a). The sum is
$2(L_n-1-\varepsilon(n{-}1))+(2-L_n+2\varepsilon(n{-}1))=L_n$,
proving (b).

For (c), the invariant~\eqref{eq:r3-Linv} gives $L_n\le 2+\varepsilon(n{-}1)$
with $L_n=3\iff\varepsilon(n{-}1)=1$, so:
if $L_n=1$ (hence $\varepsilon(n{-}1)=0$): $d(3n{-}2)=d(3n)=0$, $d(3n{-}1)=1$;
if $L_n=2$ (hence $\varepsilon(n{-}1)=0$): $d(3n{-}2)=d(3n)=1$, $d(3n{-}1)=0$;
if $L_n=3$ (hence $\varepsilon(n{-}1)=1$): $d(3n{-}2)=d(3n)=1$, $d(3n{-}1)=1$.
\end{proof}

\begin{remark}\label{rem:oscillator}
Proposition~\ref{prop:palindromic} explains the oscillatory behavior of
$a(n)/n$ for $r=3$ as a delayed-feedback mechanism.
The infinite binary word $W=d(1)d(2)d(3)\cdots$ is self-describing:
the distance between consecutive $1$s in $W$ (at positions marking the
end of each run) equals the number of $1$s in the
sliding window of size $r=3$ ending at the corresponding index.
When the local density of $1$s is high (many short runs),
the rule dictates long future runs, creating a sparse zone;
when the window reaches this sparse zone, the count of $1$s drops,
forcing short runs and recreating a dense zone.
This perpetual oscillation is the source of the non-convergence of $a(n)/n$.
\end{remark}

\begin{remark}\label{rem:laplacian}
The propagation rules \eqref{eq:Lprop1}--\eqref{eq:Lprop3}
for the run lengths satisfy the relation
\[
  L_{3n-1}+L_{3n}+L_{3n+1}
  = L_{n-1}+L_n+L_{n+1}
    - \bigl(\varepsilon(n{-}2)-2\varepsilon(n{-}1)+\varepsilon(n)\bigr).
\]
The expression $f(n{-}1)-2f(n)+f(n{+}1)$ is the discrete Laplacian
$\Delta^2 f(n)$. Hence the total run length in a triadic block
is conserved up to a correction by $\Delta^2\varepsilon(n{-}1)$:
the points where $\varepsilon=1$ act as localized sources that
redistribute mass between adjacent runs, analogous to a discrete
diffusion process.
\end{remark}

\section{A Golomb meta-structure}\label{sec:golomb_return}

The almost Golomb sequences of order $r$ were introduced as finite-memory variants
of Golomb's sequence, departing from it in growth rate, regularity, and oscillatory
behavior. Golomb's sequence nevertheless re-emerges within the almost Golomb
hierarchy through a boundary phenomenon, which we now prove: the maximal
multiplicity of the order-$r$ sequence is a value of Golomb's sequence, and the
thresholds at which it increases are governed by the partial sums of Golomb's
sequence.

The two results on which this rests, the Prefix theorem and the Domination
theorem below, were stated as conjectures in the first version of this paper
(arXiv:2604.02404v1, April 2026). They were first proved by Shaoshi
Zhou~\cite{Zhou26}, using the structural results of Section~\ref{sec:general},
and independently by the authors. The proofs given here are ours; they are
essentially the same as Zhou's.

For an almost Golomb sequence $(a(n))$ of order $r$, let $N_r(n)$ denote
the number of times $n$ appears in the sequence.
In Golomb's sequence, the self-describing property gives
$N(n)=G(n)\sim C n^{\varphi-1}$, which grows without bound.
For finite $r$, by contrast, $N_r(n)$ is bounded since $a(n)\asymp n$.
For $r=2$ this is made precise in Proposition~\ref{prop:N2_multiplicity}:
$N_2(n)\in\{1,2\}$ and $(N_2(n))$ is $2$-automatic.

Numerical computation for $r=2,\ldots,10$ yields the following maximal
multiplicities $M(r)=\sup_n N_r(n)$:
\[
\begin{array}{c|ccccccccc}
  r & 2 & 3 & 4 & 5 & 6 & 7 & 8 & 9 & 10 \\\hline
  M(r) & 2 & 3 & 3 & 3 & 3 & 4 & 4 & 4 & 5
\end{array}
\]
Let $j_k$ denote the smallest $r$ such that $M(r) \ge k$.
Numerical computation up to $r=700$ yields
\[
  j_3=3,\; j_4=7,\; j_5=10,\; j_6=13,\; j_7=17,\; j_8=21,\;
  j_9=25,\; j_{10}=30,\;\ldots
\]
with differences $j_{k+1}-j_k = 4,3,3,4,4,4,5,5,5,6,6,6,6,7,\ldots$
(starting from $k=3$).
The first gap $j_4-j_3=4$ is anomalous; from $k=4$ onwards the
difference sequence coincides exactly with Golomb's sequence,
\[
  j_{k+1}-j_k = G(k), \qquad k\ge 4,
\]
where $(G(n))_{n\ge 1}$ is Golomb's sequence as defined in
Section~\ref{sec:golomb} (OEIS \oeis{A001462}).
Equivalently, letting $S(k)=\sum_{n=1}^{k}G(n)$ denote the partial
sums of Golomb's sequence (OEIS \oeis{A001463}), one has the closed formula
\[
  j_k = S(k-1) + 2, \qquad k\ge 4.
\]
Recall that $S(k)$ is characterized by $G(S(k))=k$: it is the position
of the last occurrence of $k$ in Golomb's sequence.
These identities were first observed numerically, for all $64$ gaps up to
$r=700$; they are proved below (Theorem~\ref{thm:golomb_return}).

Since $S(k)\sim \varphi^{2-\varphi}k^\varphi/\varphi$, the threshold law gives
\[
  M(r) \;\sim\; C\,r^{\varphi-1}, \qquad r\to\infty,
\]
where $\varphi-1\approx 0.618$ is the same exponent as in
Golomb's own asymptotic $a(n)\sim C n^{\varphi-1}$.
We thus come full circle. The sequences defined by truncating Golomb's
global rule to a finite window of size $r$ carry, in their run-length
structure, the fingerprint of the very sequence they were derived from.

Numerical data are collected in Table~\ref{tab:golomb_meta}.

\begin{table}[ht]
\centering
\renewcommand{\arraystretch}{1.35}
\caption{The Golomb meta-structure.
\textbf{Top:} $M(r)=\sup_n N_r(n)$ for $r=2,\ldots,50$.
\textbf{Bottom:} threshold positions $j_k$ (smallest $r$ with $M(r)\ge k$),
their successive differences $j_{k+1}-j_k$, and the values $G(k)$ of Golomb's
sequence~\cite{Golomb66}.
The identity $j_{k+1}-j_k = G(k)$ holds for all $k\ge 4$ by
Theorem~\ref{thm:golomb_return}; it was first observed numerically
on $28$ consecutive values up to $r=200$ and $64$ values up to $r=700$.}
\label{tab:golomb_meta}
\smallskip

{\small
\begin{tabular}{c|ccccccccccccccccc}
$r$    &  2& 3& 4& 5& 6& 7& 8& 9&10&11&12&13&14&15&16&17&18 \\\hline
$M(r)$ &  2& 3& 3& 3& 3& 4& 4& 4& 5& 5& 5& 6& 6& 6& 6& 7& 7
\end{tabular}

\smallskip
\begin{tabular}{c|ccccccccccccccccc}
$r$    & 19&20&21&22&23&24&25&26&27&28&29&30&31&32&33&34&35 \\\hline
$M(r)$ &  7& 7& 8& 8& 8& 8& 9& 9& 9& 9& 9&10&10&10&10&10&11
\end{tabular}

\smallskip
\begin{tabular}{c|ccccccccccccccc}
$r$    & 36&37&38&39&40&41&42&43&44&45&46&47&48&49&50 \\\hline
$M(r)$ & 11&11&11&11&12&12&12&12&12&12&13&13&13&13&13
\end{tabular}
}

\bigskip

{\small
\begin{tabular}{c|c|cccccccccccccc}
$k$             &  3 &  4 &  5 &  6 &  7 &  8 &  9 & 10 & 11 & 12 & 13 & 14 & 15 & 16 \\\hline
$j_k$           &  3 &  7 & 10 & 13 & 17 & 21 & 25 & 30 & 35 & 40 & 46 & 52 & 58 & 64 \\\hline
$j_{k+1}{-}j_k$ &  4 &  3 &  3 &  4 &  4 &  4 &  5 &  5 &  5 &  6 &  6 &  6 &  6 &  7 \\\hline
$G(k)$          &  2 &  3 &  3 &  4 &  4 &  4 &  5 &  5 &  5 &  6 &  6 &  6 &  6 &  7 \\
\end{tabular}

\smallskip
\begin{tabular}{c|cccccccccccccc}
$k$             & 17 & 18 & 19 & 20 & 21 & 22 & 23 & 24 & 25 & 26 & 27 & 28 & 29 & 30 \\\hline
$j_k$           & 71 & 78 & 85 & 92 &100 &108 &116 &124 &133 &142 &151 &160 &169 &179 \\\hline
$j_{k+1}{-}j_k$ &  7 &  7 &  8 &  8 &  8 &  8 &  9 &  9 &  9 &  9 &  9 & 10 & 10 & 10 \\\hline
$G(k)$          &  7 &  7 &  7 &  8 &  8 &  8 &  8 &  9 &  9 &  9 &  9 &  9 & 10 & 10 \\
\end{tabular}
}
\end{table}

\begin{theorem}[Threshold law]\label{thm:golomb_return}
  For all $k\ge 4$, the threshold $j_k$ (smallest order $r$ at which
  $M(r)=\sup_n N_r(n)$ first reaches $k$) satisfies
  \[
    j_{k+1}-j_k = G(k) \qquad \text{and} \qquad j_k = S(k-1)+2,
  \]
  where $G$ is Golomb's sequence and $S(k)=\sum_{n=1}^{k}G(n)$.
\end{theorem}

The proof occupies the rest of this section: it is
Corollary~\ref{cor:threshold}, which combines the Prefix theorem and the
Domination theorem.

\subsection{Proof of the threshold law}

Set $S_m^{(r)} = \sum_{i=0}^{r-1} a_r(m-i)$ (with $a_r(k)=0$ for $k\le 0$),
so that $a_r(S_m^{(r)})=m$ by the defining equation.
The following lemma provides the run-length identity on which everything rests.

\begin{lemma}\label{lem:run_identity}
  For all $r\ge 2$ and $m\ge 3$,
  \begin{equation}\label{eq:run_diff}
    N_r(m) = a_r(m+1) - a_r(m+1-r).
  \end{equation}
\end{lemma}

\begin{proof}
  By Theorem~\ref{thm:global_run}, the run of value $m$ is exactly
  $[S_m^{(r)},S_{m+1}^{(r)}-1]$, so $N_r(m)=S_{m+1}^{(r)}-S_m^{(r)}$.
  Telescoping: $S_{m+1}^{(r)}-S_m^{(r)}=\sum_{j=0}^{r-1}a_r(m+1-j)-\sum_{j=0}^{r-1}a_r(m-j)
  =a_r(m+1)-a_r(m+1-r)$.
\end{proof}

The first step concerns the value of a single term:
for $r\ge 4$, Lemma~\ref{lem:run_identity} at $m=r-1$ gives
$N_r(r-1)=a_r(r)-a_r(0)=a_r(r)$.

\begin{theorem}[Prefix theorem]\label{thm:prefix}
For every $r\ge 3$ and every $n$ with $3\le n\le r$,
\[
  a_r(n) = G(n-1).
\]
In particular $a_r(r)=G(r-1)$.
\end{theorem}

\begin{proof}
Write $S(k)=\sum_{n\le k}G(n)$, with $S(0)=0$, so that the run of the
value $v$ in Golomb's sequence is exactly $[S(v-1)+1,\,S(v)]$.

The starting point is that the run-length formula (Equation~\ref{EQ:runlength})
degenerates on the initial range. For $m\le r-1$ we have $m+1-r\le 0$, so
$a_r(m+1-r)=0$ and
\begin{equation}\label{eq:prefix-selfdesc}
  L_m \;=\; a_r(m+1)-a_r(m+1-r) \;=\; a_r(m+1) \qquad (3\le m\le r-1).
\end{equation}
Thus on this range the value $m$ occurs exactly $a_r(m+1)$ times, a
self-describing property shifted by one index.

We claim that
\begin{equation}\label{eq:prefix-anchor}
  S_v = S(v-1)+2 \qquad (3\le v\le r).
\end{equation}
By Lemma~\ref{LEM:InitValues}, $S_3=5$ for $r\ge 3$, and $S(2)+2=5$, which
is the case $v=3$.

Assume Equation~\ref{eq:prefix-anchor} for all $v\le m$, where $3\le m\le r-1$.
By Theorem~\ref{thm:global_run}(b) the run of $v$ is $[S_v,S_{v+1}-1]$, so
for $5\le n\le S_m-1$ the induction hypothesis gives
\begin{align*}
  a_r(n)=v &\iff S(v-1)+2\le n\le S(v)+1 \\
           &\iff S(v-1)+1\le n-1\le S(v) \iff G(n-1)=v,
\end{align*}
that is, $a_r(n)=G(n-1)$ on that range. The value $a_r(m+1)$ is therefore
already determined, provided $m+1\le S_m-1$, which is equivalent to
$m\le S(m-1)$; this holds since $G(k)\ge 2$ for $k\ge 2$ gives
$S(m-1)\ge 1+2(m-2)=2m-3\ge m$ for $m\ge 3$. Hence $a_r(m+1)=G(m)$, and
Equation~\ref{eq:prefix-selfdesc} yields $L_m=G(m)$. Consequently
\[
  S_{m+1}=S_m+L_m=S(m-1)+2+G(m)=S(m)+2,
\]
which is Equation~\ref{eq:prefix-anchor} at rank $m+1$ and closes the induction.

From Equation~\ref{eq:prefix-anchor} we obtain $a_r(n)=G(n-1)$ for all
$3\le n\le S_r-1$, the cases $n=3,4$ following from
Lemma~\ref{LEM:InitValues} together with $G(2)=G(3)=2$. Finally
$S_r-1=S(r-1)+1\ge r+1$, so the range covers $n=r$.
\end{proof}

\begin{corollary}\label{cor:prefix-boundary}
For every $r\ge 4$, the boundary multiplicity satisfies
\[
  N_r(r-1)=a_r(r)=G(r-1).
\]
More generally $N_r(m)=G(m)$ for $3\le m\le r-1$, so the multiplicity
increases along the whole initial range and attains its maximum there
at $m=r-1$.
\end{corollary}

\begin{proof}
Immediate from Lemma~\ref{lem:run_identity} at $m=r-1$ together with
Theorem~\ref{thm:prefix}, and from Equation~\ref{eq:prefix-selfdesc} for the
general statement.
\end{proof}

The second step shows that no later run exceeds the boundary run.
For $v\ge 1$ write $F_r(v)$ for the position of the first occurrence of the
value $v$ in $a_r$. By Theorem~\ref{thm:global_run}(b), $F_r(v)=S_v$ for
$v\ge 3$, while $F_r(1)=1$ and, by Lemma~\ref{LEM:InitValues}, $F_r(2)=2$
whereas $S_2=3$. For positions $p<q$, the difference $a_r(q)-a_r(p)$ is the
number of values $v$ with $p<F_r(v)\le q$, that is, the number of run starts
in the interval $(p,q]$.

\begin{theorem}[Domination theorem]\label{thm:domination}
For every $r\ge 5$, every run of $a_r$ has length at most $G(r-1)$.
Consequently
\[
  M(r)=\sup_{n\ge 1}N_r(n)=N_r(r-1)=G(r-1).
\]
\end{theorem}

\begin{proof}
Write $k:=G(r-1)$; since $r-1\ge 4$ we have $k\ge 3$. We prove by strong
induction on $m\ge 1$ that the run of the value $m$ has length at most $k$.

The run of $1$ is $\{1\}$, and by Lemma~\ref{LEM:InitValues} the run of $2$
is $[2,S_3-1]=[2,4]$, of length $3\le k$. For $3\le m\le r-1$,
Corollary~\ref{cor:prefix-boundary} gives $N_r(m)=G(m)\le G(r-1)=k$.

Let $m\ge r$ and assume the claim for all values below $m$. By
Lemma~\ref{lem:run_identity} the run of $m$ has length
$L_m=a_r(m+1)-a_r(m+1-r)$, which is the number of run starts in the window
$W=[m-r+2,\,m+1]$. Suppose $L_m\ge k+1$. The run starts in $W$ belong to
consecutive values; let $F_r(u)<\dots<F_r(u+k)$ be the first $k+1$ of them.
Since $F_r(u)\ge m-r+2\ge 2$ and $F_r(1)=1$, we have $u\ge 2$, and since all
lie in $W$,
\begin{equation}\label{eq:dom-span}
  F_r(u+k)-F_r(u)\le r-1.
\end{equation}
Note that $u+k\ge 5$, so $F_r(u+k)=S_{u+k}$, and that
$F_r(u)\le S_u$ in all cases, with equality for $u\ge 3$. Hence
\begin{equation}\label{eq:dom-reduce}
  F_r(u+k)-F_r(u)\;\ge\;S_{u+k}-S_u .
\end{equation}
We show $S_{u+k}-S_u\ge r$, which contradicts Equation~\ref{eq:dom-span}.

\emph{Case $u+k\le r$.} All indices $u,\dots,u+k$ lie in the range where
$S_v=S(v-1)+2$: this is Equation~\ref{eq:prefix-anchor} for $3\le v\le r$, and
$S_2=3=S(1)+2$ by Lemma~\ref{LEM:InitValues}. Hence
\[
  S_{u+k}-S_u=\sum_{v=u}^{u+k-1}G(v)\;\ge\;\sum_{v=2}^{k+1}G(v)
  \;=\;S(k+1)-1\;\ge\;r,
\]
using $u\ge 2$ and the monotonicity of $G$ for the first inequality, and for
the last the fact that $G(r-1)=k$ places $r-1$ in the run of $k$, so
$r-1\le S(k)$, while $G(k+1)\ge 2$ gives $S(k+1)\ge S(k)+2$.

\emph{Case $u+k\ge r$.} From the definition of the sliding sums,
\begin{equation}\label{eq:dom-identity}
  S_{u+k}-S_u=\sum_{j=0}^{r-1}\bigl[a_r(u+k-j)-a_r(u-j)\bigr],
\end{equation}
and it suffices to show that every bracket is at least $1$.
If $u-j\le 0$, the bracket equals $a_r(u+k-j)\ge 1$, since
$u+k-j\ge u+k-r+1\ge 1$. If $p:=u-j\ge 1$, let $w:=a_r(p)$. Since
$F_r(u)\le m+1$ we have $u=a_r(F_r(u))\le a_r(m+1)\le m$ by
Lemma~\ref{lem:subdiag}, hence $w\le a_r(u)\le a_r(m)\le m-1$.
If $w=1$ then $p=1$ and $a_r(1+k)\ge a_r(4)=2$. If $w=2$ then $p\in[2,4]$,
so $p+k\ge 5$ and $a_r(p+k)\ge 3$. If $w\ge 3$, the induction hypothesis
gives $L_w\le k$; the run of $w$ is $[S_w,S_w+L_w-1]$ with $S_w\le p$, so it
ends before $p+k$, and $a_r(p+k)\ge w+1$. In every case
$a_r(u+k-j)-a_r(u-j)\ge 1$, so $S_{u+k}-S_u\ge r$.

This proves the bound on run lengths. Since $N_r(1)=1$, $N_r(2)=3$ and
$N_r(m)=L_m$ for $m\ge 3$, we get $M(r)\le k=N_r(r-1)$, and the reverse
inequality is trivial.
\end{proof}

\begin{remark}
The argument explains the exception $r=4$: there $k=G(3)=2$, while the
run of the value $2$ has length $3$, and the step $w=2$ fails, since
$a_4(2+k)=a_4(4)=2$. Indeed $M(4)=3>G(3)$.
Numerically the bound of the theorem is also attained in the stationary
range: for every $r$ we have tested with $5\le r\le 200$, the value $G(r-1)$
is reached by $N_r(m)$ infinitely often, at roughly regularly spaced indices
$m\ge r$. So the boundary value is not an upper bound that the stationary
regime stays below, but a value the two regimes share.
\end{remark}

\begin{corollary}\label{cor:threshold}
For every $k\ge 4$, $j_k=S(k-1)+2$ and $j_{k+1}-j_k=G(k)$.
\end{corollary}

\begin{proof}
By Theorem~\ref{thm:domination}, $M(r)=G(r-1)$ for $r\ge 5$, and
$S(k-1)+2\ge 7$ for $k\ge 4$, so $j_k$ is the least $r\ge 5$ with
$G(r-1)\ge k$, namely $r-1=S(k-1)+1$. The difference is
$S(k)-S(k-1)=G(k)$.
\end{proof}

\begin{corollary}\label{cor:lmax}
For every $r\ge 5$,
\[
  L_{\max}(r)=\limsup_{n\to\infty}\frac{a_r(n)}{n}\;\le\;\frac{G(r-1)}{r}.
\]
In particular $L_{\max}(r)\to 0$ as $r\to\infty$, and
$L_{\max}(r)=O(r^{\varphi-2})$ with $\varphi-2\approx-0.382$.
\end{corollary}

\begin{proof}
By Theorem~\ref{thm:domination} and Equation~\ref{EQ:runlength},
$a_r(n+r)-a_r(n)=L_{n+r-1}\le G(r-1)$ for every $n\ge 1$. Hence
$a_r(n)\le a_r(n_0)+\lfloor (n-n_0)/r\rfloor\,G(r-1)$ for $n\ge n_0$, and
dividing by $n$ gives the bound. The rate follows from
$G(n)\sim C\,n^{\varphi-1}$.
\end{proof}

\section{Open questions and perspectives}\label{sec:open}

With the Golomb meta-structure established, several structural and analytic questions remain open.
\begin{enumerate}
  \item \textbf{The singular limit and phase transition ($r \to \infty$).}
        For any finite $r$, the sequence exhibits oscillatory linear growth
        $a(n) \asymp n$, whereas Golomb's sequence (the formal limit $r\to\infty$)
        grows sublinearly as $G(n) \sim C n^{\varphi-1}$.
        How does the sequence undergo this phase transition?
        Letting $L_{\max}(r) = \limsup_{n\to\infty} a(n)/n$,
        Corollary~\ref{cor:lmax} gives $L_{\max}(r)\le G(r-1)/r$ for $r\ge 5$,
        so $L_{\max}(r)\to 0$ at rate $O(r^{\varphi-2})$. Two questions remain.
        Numerical evidence suggests that $L_{\max}(r)$ is strictly decreasing
        (values $3/4$, $5/8$, $25/48$, $93/200$ for $r=2,3,4,5$); is it?
        And is the exponent $\varphi-2$ sharp, that is, does
        $r^{2-\varphi}L_{\max}(r)$ converge to a positive constant?

  \item \textbf{Explicit denesting structure.}
        Theorem~\ref{thm:structural_regularity} proves $r$-regularity for all $r$.
        However, the DFAO produced is large. For each $r$, does the almost Golomb
        sequence admit an explicit denesting $a(rn+i)=T_r(n+s_i)+\delta_i+\varepsilon_i(n)$
        with $\varepsilon_i$ $r$-automatic and satisfying compact Boolean recurrences,
        as established for $r\le 5$? A proof of this finer structure would explain why
        \texttt{Walnut} produces automata with $O(r^2)$ states rather than the theoretical bound.

  \item \textbf{Ces\`aro means.}
        For $r=2$, Theorem~\ref{thm:ratio2} shows that the Ces\`aro means of
        $a(n)/n$ do not converge. For $r \ge 3$, do the block averages
        $\frac{1}{N}\sum_{n=1}^N a(n)/n$ converge, or do they exhibit
        fractal fluctuations periodic in $\log_r N$?

  \item \textbf{Automaticity of the full multiplicity sequence.}
        For $r=2$, $(N_2(n))$ is $2$-automatic
        (Proposition~\ref{prop:N2_multiplicity}). Does the $r$-regularity
        of $(a(n))$ guarantee that $(N_r(n))_{n\ge 1}$ is always $r$-automatic?

  \item \textbf{Corrector complexity as a function of $r$.}
For $r=3$ the correction data is a single binary sequence, recognised by a
minimal DFAO with $7$ states. For $r=4$ it is a single sequence over a
four-letter alphabet and for $r=5$ a single sequence over a three-letter
alphabet, with minimal DFAOs of comparable size. How does the minimal
automaton size grow with $r$, and does the arithmetic nature of $r$ play any
role? Three orders are too few to separate a genuine dependence on $r$ from
the particular behavior of small cases; the first informative comparison
would be $r=6$ against $r=7$.

  \item \textbf{The gap-$s$ variant.}
        Fix a gap $s\ge 1$ and consider the sequence defined by
        $a\bigl(a(n)+a(n-s)\bigr)=n$ for $n\ge s+1$, with $a(k)=0$ for $k\le 0$
        and $a(1)=1$.

        For $s=1$ this is the almost Golomb sequence of order $r=2$ studied in
        Section~\ref{sec:r2}.

        For $s=2$, the sequence begins
        $1,2,2,3,3,4,5,6,6,7,7,8,8,\ldots$
        and satisfies $a(n+1)-a(n)\in\{0,1\}$ for all $n\ge 1$.
        One can show that the sequence admits an exact divide-and-conquer
        structure modulo a binary correction sequence $\varepsilon(n)\in\{0,1\}$,
        whose support is
        \[
          \mathcal{I} = \{5,12\}\;\cup\;\bigcup_{k\ge 2}\mathcal{I}_k,
          \qquad
          \mathcal{I}_k = \bigl[5\cdot 2^k+2,\;6\cdot 2^k-2\bigr]\cap(4\Z+2),
        \]
        with $|\mathcal{I}_k|=2^{k-2}$, satisfying $\mathcal{I}_k=\{6\cdot 2^k-2\}\cup(2\cdot\mathcal{I}_{k-1}-2)$.
        Two families of exact values hold:
        \[
          a(7\cdot 2^k) = 5\cdot 2^k\ (k\ge 1),\quad
          a(5\cdot 2^k) = 7\cdot 2^{k-1}\ (k\ge 2),\quad
          a(3\cdot 2^k) = 17\cdot 2^{k-3}\ (k\ge 3),
        \]
        and $a(n)/n$ oscillates between the exact limits
        \[
          \liminf_{n\to\infty}\frac{a(n)}{n} = \frac{7}{10},\qquad
          \limsup_{n\to\infty}\frac{a(n)}{n} = \frac{5}{7}.
        \]
        This should be compared with the Stern diatomic sequence~\cite[Chap.\,16]{AS03}
        (\oeis{A002487}), which satisfies $s(2n)=s(n)$, $s(2n+1)=s(n)+s(n+1)$,
        is $2$-regular with differences in $\{0,1\}$, and admits no correction
        sequence because its denesting is exact.

        For $s=3$, the sequence $1,3,3,4,4,5,5,6,7,8,8,$\linebreak$9,10,11,11,\ldots$
        has differences in $\{0,1,2\}$ and numerical investigation reveals
        no stable divide-and-conquer recurrence at any base.
        This case appears to escape the $2$-regular framework entirely.

        It is an open problem to determine for which values of $s$ the
        gap-$s$ almost Golomb sequence is $2$-regular, and whether $s\in\{1,2\}$
        are the only cases admitting an exact correction-set description.
\end{enumerate}

\appendix
\section{Illustrative tetradic derivations}\label{app:r4-deriv}

This appendix presents a sample derivation from
Theorem~\ref{thm:r4eps}, to display the elimination mechanism
that identifies the Boolean recurrences.
Recall that $S_m = a(m-3)+a(m-2)+a(m-1)+a(m)$ for $r=4$ from Equation~\ref{EQ:windowSum}. Adjusting the argument for the denesting formulas (\ref{eq:r4a}--\ref{eq:r4d}), we get for $m\ge 5$, 
\begin{align*}
  a(4m)   &= S_m+1+\varepsilon_0(m), &\quad
  a(4m+1) &= S_{m+1}+\varepsilon_1(m),\\
  a(4m+2) &= S_{m+1}+\varepsilon_2(m), &\quad
  a(4m+3) &= S_{m+2}-1+\varepsilon_3(m).
\end{align*}

\subsection*{Derivation for $\varepsilon_1(4m+1)=1-\varepsilon_2(m)$ (Equation \ref{eq:e1_1})}

Definition \ref{eq:r4def} implies $a(S_{4m+1})=4m+1$ where $S_{4m+1} = a(4m+1)+a(4m)+a(4m-1)+a(4m-2)$.

Expanding the terms on the right hand side, we get
\begin{align*}
  a(4m+1) &= S_{m+1}+\varepsilon_1(m),\\
  a(4m)   &= S_{m}+1+\varepsilon_0(m),\\
  a(4m-1) &= S_{m+1}-1+\varepsilon_3(m-1),\\
  a(4m-2) &= S_{m}+\varepsilon_2(m-1).
\end{align*}
Hence $S_{4m+1} = 2S_m+2S_{m+1} + (\varepsilon_1(m)+\varepsilon_0(m)+\varepsilon_3(m-1)+\varepsilon_2(m-1))$.
Writing $E_m := \varepsilon_0(m)+\varepsilon_1(m)+\varepsilon_2(m-1)+\varepsilon_3(m-1)\in\{0,1,2,3,4\}$
and $F_m := a(m-2)+a(m-1)+a(m)$, one has $2S_m+2S_{m+1} = 4F_m+2(a(m-3)+a(m+1))$,
so $S_{4m+1} \equiv 2(a(m-3)+a(m+1))+E_m \pmod{4}$.

The residue of $S_{4m+1}$ modulo $4$ determines which of Equations \ref{eq:r4a}--\ref{eq:r4d}
applies at $S_{4m+1}$, and therefore, which $\varepsilon_j$ appears when one writes
$a(S_{4m+1})=4m+1$. In each of the four sub-cases, imposing $a(S_{4m+1})=4m+1$
and using the known values of $\varepsilon_j(m+\delta)$ for $|\delta|\le 1$ yields
the single Boolean identity
\[
  \varepsilon_1(4m+1) = 1-\varepsilon_2(m).
\]
\subsection*{Note for obtaining the remaining recurrences}
Each of the fifteen remaining identities within
(\ref{eq:e0_0}--\ref{eq:e3_3}) is obtained by the same procedure: expand
$S_{4m+i}$ using (\ref{eq:r4a}--\ref{eq:r4d}), determine which residue
formula applies at $S_{4m+i}$, and impose $a(S_{4m+i})=4m+i$. The result is
a Boolean polynomial in $\varepsilon_j(m+\delta)$ with $|\delta|\le 1$,
which simplifies to the stated recurrence. All sixteen identities are
certified by \texttt{Walnut}.
\section{Illustrative pentadic derivations}\label{app:r5-deriv}
Recall from Section~\ref{sec:r5},
\[
  S_n := a(n)+a(n-1)+a(n-2)+a(n-3)+a(n-4).
\]
The denesting formulas \eqref{eq:r5a0}--\eqref{eq:r5a2} read
\[
  a(5m) = S_m+2+\varepsilon(m),\quad
  a(5m+1) = S_{m+1}+1-\varepsilon(m),\quad
  a(5m+2) = S_{m+2}-1-\eta(m).
\]
\subsection*{Derivation of Equation~\ref{eq:r5a1}: $a(5m+1)=S_{m+1}+1-\varepsilon(m)$}
The defining equation at $n=5m+1$ reads
$a(S_{5m+1})=5m+1$ where $S_{5m+1}=a(5m+1)+a(5m)+a(5m-1)+a(5m-2)+a(5m-3)$.
Expand the four terms $a(5m),a(5m-1),a(5m-2),a(5m-3)$ using
Equations~\ref{eq:r5a0} and~\ref{eq:r5a2} (at $n=m-1$), and the neighboring residue
formulas. The deterministic part of $S_{5m+1}$ equals $S_{m+1}+(1-\varepsilon(m))+S_m$,
and after the second expansion via $a(S_{5m+1})=5m+1$ all $a(\cdot)$-terms cancel,
yielding Equation~\ref{eq:r5a1}.
\subsection*{How to obtain the remaining recurrences in
Theorem~\ref{thm:r5eps}}

Each of the ten remaining identities
\eqref{eq:r5eps0}--\eqref{eq:r5eta13} is obtained by the same procedure:
apply the defining equation at $n=5m+d$ for the appropriate $d$,
expand the window sum via the five residue formulas, and cancel the
$a(\cdot)$-terms. The result is always a function of $(\varepsilon(m),\eta(m))$
or $(\varepsilon(m-1),\eta(m-1))$, which simplifies to the stated recurrence.

\section{Pentadic correction tables}\label{app:r5-tables}

Write $U(n)=(\varepsilon(n),\eta(n))\in\{(0,0),(1,0),(0,1)\}$.
For $m\ge 3$ the only transitions $(U(m-1),U(m))$ that occur are:
\[
(0,0)\to(0,0),\quad (0,0)\to(0,1),\quad (0,1)\to(1,0),\quad
(1,0)\to(0,0),\quad (1,0)\to(0,1).
\]

\subsection*{Table for $\theta$}

The corrector $\theta(5m+d)\in\{0,1\}$ is determined by $(U(m-1),U(m))$ and
$d\in\{0,1,2,3,4\}$:
\[
\begin{array}{c|ccccc}
(U(m-1),U(m)) & d=0&1&2&3&4\\\hline
((0,0),(0,0)) & 1&1&1&1&1\\
((1,0),(0,0)) & 1&1&1&1&1\\
((0,0),(0,1)) & 0&1&0&1&0\\
((1,0),(0,1)) & 0&1&0&1&0\\
((0,1),(1,0)) & 1&0&1&0&1
\end{array}
\]

\subsection*{Table for $\varepsilon_4$}

For $d\in\{0,1,2,3\}$, set $\sigma(m)=\varepsilon(m)+\eta(m)$; then
\[
  \varepsilon_4(5m+d) =
  \begin{cases}
    2-\sigma(m)-\varepsilon(m), & d\in\{0,2\},\\
    2-\sigma(m)-\eta(m), & d\in\{1,3\}.
  \end{cases}
\]
For $d=4$, the value is determined by $(U(m),U(m+1))$:
\[
  \varepsilon_4(5m+4) =
  \begin{cases}
    2,& (U(m),U(m+1))=((0,0),(0,0)),\\
    1,& (U(m),U(m+1))\in\bigl\{((0,0),(0,1)),((0,1),(1,0)),((1,0),(0,0))\bigr\},\\
    0,& (U(m),U(m+1))=((1,0),(0,1)).
  \end{cases}
\]
All entries are consistent with the $5$-automatic structure certified
by \texttt{Walnut} (Appendix~\ref{app:dfao}) and have been verified independently
for all $n\le 2000$.

\section{A menagerie of nested self-referential sequences}\label{app:nested}

The almost Golomb sequences studied in this paper belong to a broader landscape of
self-referential sequences admitting a denesting into base-$k$
divide-and-conquer formulas. We collect here the most relevant examples
for comparison, distinguishing two types.

\subsection*{Monotone sequences}

These are defined greedily: $a(n)$ is the \emph{smallest} integer $\ge a(n-1)$
consistent with the defining constraint. They are nondecreasing by construction,
unbounded, and hence $k$-regular (not $k$-automatic) for appropriate $k$;
see~\cite{AS92} and~\cite[Chapter~16]{AS03}.

\medskip
{\small
\begin{center}
\renewcommand{\arraystretch}{1.5}
\begin{tabular}{c|p{9cm}|c}
OEIS & \multicolumn{1}{c|}{Implicit rule \& denesting} & $k$-reg. \\\hline
\oeis{A003605} & $a(a(n))=3n$;\;
  $a(3n)=3a(n)$, $a(3n{+}1)=2a(n){+}a(n{+}1)$,
  $a(3n{+}2)=a(n){+}2a(n{+}1)$ & $3$ \\
\oeis{A079000} & $a(a(n))=2n+3$;\; $2$-adic formulas~\cite{CSV03,ARS05} & $2$ \\
\oeis{A080637} & $a(a(n))=2n+1$;\; $2$-adic formulas~\cite{CSV03,ARS05} & $2$ \\
\oeis{A394217} & $a(a(n){+}a(n{-}1))=n$;\;
  $a(2n)=a(n){+}a(n{+}1){-}1$,\; $a(2n{+}1)=a(n){+}a(n{+}1)$ & $2$ \\
\oeis{A394218} & $a(a(n){+}a(n{-}1){+}a(n{-}2))=n$;\;
  $3$-adic with correction $\one_\Ical$, \S\ref{sec:r3} & $3$ \\
\end{tabular}
\end{center}
}
\medskip

The last two rows are the almost Golomb sequences of orders $2$ and $3$ studied in
this paper, which generalize the $a(a(n))=kn$ family (studied systematically
in~\cite{ARS05}) by replacing the single
nested index $a(n)$ with a sliding window sum.
A related family of \emph{meta-automatic} sequences, combining
nested meta-Fibonacci recurrences with digit-based rules, is studied
in~\cite{CC25}.

\subsection*{Earliest (non-monotone) sequences}

These are defined differently: $a(n)$ is the \emph{smallest positive integer not
yet used} consistent with the constraint. They are \emph{not} nondecreasing in
general, and their structure is quite different from the monotone case.

\medskip
{\small
\begin{center}
\renewcommand{\arraystretch}{1.5}
\begin{tabular}{c|p{9cm}|c}
OEIS & \multicolumn{1}{c|}{Implicit rule \& denesting} & Type \\\hline
\oeis{A002516} & $a(a(n))=2n$, earliest;\;
  $a(4n)=2a(2n)$, $a(4n{+}1)=4n{+}3$, $a(4n{+}3)=8n{+}2$ & $2$-aut. \\
\oeis{A002517} & $a(a(n))=3n$, earliest;\;
  $a(3n)=3a(n)$, $a(3n{+}1)=3n{+}2$, $a(3n{+}2)=9n{+}3$ & $3$-aut. \\
\end{tabular}
\end{center}
}
\medskip

The earliest sequences take values in a rearrangement of the integers,
whereas the monotone sequences (and the almost Golomb sequences) are nondecreasing.
This difference accounts for the fact that the earliest sequences are $k$-automatic
(bounded values modulo finite-state processing) while the monotone ones are merely
$k$-regular (unbounded but finitely generated kernel).


\section{DFAO data for the correction sequences}\label{app:dfao}

This appendix provides transition tables of minimal DFAOs
(deterministic finite automata with output) realising the correction sequences
for $r=4$ and $r=5$.
These automata were identified computationally from the initial segments of the correction sequences. They are the inputs on which the \texttt{Walnut} predicates of Sections~\ref{sec:r4} and~\ref{sec:r5} are run, and from which the corrector patterns along geometric progressions are read off in Section~\ref{sec:nonconv}, so the certifications there depend on exactly the transition tables printed below.
\subsection*{The case $r=4$}

The four correction sequences $\varepsilon_0,\varepsilon_1,\varepsilon_2,\varepsilon_3$
of Theorem~\ref{thm:r4eps} are each recognised by an explicit DFAO reading
base-$4$ expansions.

\noindent\begin{minipage}{\linewidth}
\textbf{DFAO \texttt{C0} for $\varepsilon_0$ (base $4$, 13 states, initial state $0$)}

\begin{center}
\renewcommand{\arraystretch}{1.1}
\begin{tabular}{c|c|cccc}
State & out & $d=0$ & $d=1$ & $d=2$ & $d=3$ \\\hline
0 & $0$ & 0 & 1 & 2 & 2 \\
1 & $0$ & 3 & 4 & 5 & 6 \\
2 & $0$ & 6 & 6 & 6 & 6 \\
3 & $0$ & 7 & 8 & 9 & 4 \\
4 & $0$ & 5 & 7 & 8 & 9 \\
5 & $1$ & 4 & 5 & 6 & 6 \\
6 & $1$ & 6 & 6 & 6 & 6 \\
7 & $1$ & 7 & 8 & 9 & 4 \\
8 & $0$ & 5 & 7 & 8 & 10 \\
9 & $1$ & 11 & 12 & 9 & 4 \\
10 & $1$ & 11 & 12 & 10 & 11 \\
11 & $0$ & 12 & 10 & 11 & 12 \\
12 & $0$ & 10 & 11 & 12 & 10
\end{tabular}
\end{center}
\end{minipage}
\bigskip

\noindent\begin{minipage}{\linewidth}
\textbf{DFAO \texttt{C1} for $\varepsilon_1$ (base $4$, 11 states, initial state $0$)}

\begin{center}
\renewcommand{\arraystretch}{1.1}
\begin{tabular}{c|c|cccc}
State & out & $d=0$ & $d=1$ & $d=2$ & $d=3$ \\\hline
0 & $0$ & 0 & 1 & 2 & 2 \\
1 & $0$ & 3 & 4 & 5 & 2 \\
2 & $0$ & 2 & 2 & 2 & 2 \\
3 & $0$ & 3 & 6 & 7 & 4 \\
4 & $1$ & 5 & 3 & 6 & 7 \\
5 & $0$ & 4 & 5 & 2 & 2 \\
6 & $1$ & 5 & 3 & 6 & 8 \\
7 & $0$ & 9 & 10 & 7 & 4 \\
8 & $0$ & 9 & 10 & 8 & 9 \\
9 & $0$ & 10 & 8 & 9 & 10 \\
10 & $1$ & 8 & 9 & 10 & 8
\end{tabular}
\end{center}
\end{minipage}
\bigskip

\noindent\begin{minipage}{\linewidth}
\textbf{DFAO \texttt{C2} for $\varepsilon_2$ (base $4$, 12 states, initial state $0$)}

\begin{center}
\renewcommand{\arraystretch}{1.1}
\begin{tabular}{c|c|cccc}
State & out & $d=0$ & $d=1$ & $d=2$ & $d=3$ \\\hline
0 & $0$ & 0 & 1 & 2 & 2 \\
1 & $0$ & 3 & 4 & 5 & 6 \\
2 & $0$ & 6 & 6 & 6 & 6 \\
3 & $0$ & 3 & 7 & 8 & 4 \\
4 & $1$ & 5 & 3 & 7 & 8 \\
5 & $1$ & 4 & 5 & 6 & 6 \\
6 & $1$ & 6 & 6 & 6 & 6 \\
7 & $1$ & 5 & 3 & 7 & 9 \\
8 & $0$ & 10 & 11 & 8 & 4 \\
9 & $0$ & 10 & 11 & 9 & 10 \\
10 & $1$ & 11 & 9 & 10 & 11 \\
11 & $1$ & 9 & 10 & 11 & 9
\end{tabular}
\end{center}
\end{minipage}
\bigskip

\noindent\begin{minipage}{\linewidth}
\textbf{DFAO \texttt{C3} for $\varepsilon_3$ (base $4$, 12 states, initial state $0$)}

\begin{center}
\renewcommand{\arraystretch}{1.1}
\begin{tabular}{c|c|cccc}
State & out & $d=0$ & $d=1$ & $d=2$ & $d=3$ \\\hline
0 & $0$ & 0 & 1 & 2 & 2 \\
1 & $0$ & 3 & 4 & 5 & 2 \\
2 & $0$ & 2 & 2 & 2 & 2 \\
3 & $0$ & 6 & 7 & 8 & 4 \\
4 & $0$ & 5 & 6 & 7 & 8 \\
5 & $0$ & 4 & 5 & 2 & 2 \\
6 & $1$ & 6 & 7 & 8 & 4 \\
7 & $0$ & 5 & 6 & 7 & 9 \\
8 & $1$ & 10 & 11 & 8 & 4 \\
9 & $1$ & 10 & 11 & 9 & 10 \\
10 & $1$ & 11 & 9 & 10 & 11 \\
11 & $0$ & 9 & 10 & 11 & 9
\end{tabular}
\end{center}
\end{minipage}
\bigskip

\subsection*{The case $r=5$}

The pair $U(n)=(\varepsilon(n),\eta(n))$ of Theorem~\ref{thm:r5eps} takes values
in $\{(0,0),(1,0),(0,1)\}$ and is recognised by an explicit DFAO reading
base-$5$ expansions (encoded as $0,1,2$ respectively).

\noindent\begin{minipage}{\linewidth}
\textbf{DFAO \texttt{C5} for $U=(\varepsilon,\eta)$ (base $5$, 9 states, initial state $0$)}

\begin{center}
\renewcommand{\arraystretch}{1.1}
\begin{tabular}{c|c|ccccc}
State & out & $d=0$ & $d=1$ & $d=2$ & $d=3$ & $d=4$ \\\hline
0 & $0$ & 0 & 1 & 2 & 3 & 8 \\
1 & $0$ & 6 & 4 & 4 & 4 & 4 \\
2 & $0$ & 4 & 4 & 5 & 8 & 7 \\
3 & $2$ & 8 & 7 & 8 & 7 & 8 \\
4 & $0$ & 4 & 4 & 4 & 4 & 4 \\
5 & $2$ & 4 & 5 & 8 & 7 & 8 \\
6 & $0$ & 8 & 6 & 4 & 4 & 4 \\
7 & $2$ & 8 & 7 & 8 & 7 & 8 \\
8 & $1$ & 7 & 8 & 7 & 8 & 7
\end{tabular}
\end{center}
\end{minipage}
\bigskip

\section*{Statement on the use of AI}

During the preparation of this work, the authors used Claude (Anthropic) to
assist with language editing, manuscript organization, the drafting and
checking of mathematical arguments, and numerical and symbolic verification.
The authors are responsible for the final content.

\section*{Acknowledgments}

BC thanks Jean-Paul Allouche for valuable comments on the case $r=2$
and on a preliminary version of this article.
\providecommand{\bysame}{\leavevmode\hbox to3em{\hrulefill}\thinspace}
\providecommand{\MR}{\relax\ifhmode\unskip\space\fi MR }
\providecommand{\MRhref}[2]{%
  \href{http://www.ams.org/mathscinet-getitem?mr=#1}{#2}
}
\providecommand{\href}[2]{#2}

\end{document}